\documentclass[11pt,a4paper,reqno]{amsart}
\usepackage[a4paper,top=3cm,bottom=3cm,inner=3cm,outer=3cm,includehead,includefoot]{geometry}
\usepackage{amsmath,amssymb,amsthm,calc,verbatim,mathrsfs,paralist,tikz}
\usetikzlibrary{shapes.misc,calc,intersections,patterns}

\newtheorem{theorem}{Theorem}
\newtheorem{lemma}[theorem]{Lemma}

\newtheorem{corollary}[theorem]{Corollary}

\def\E{\mathbb{E}}
\def\N{\mathbb{N}}
\def\P{\mathbb{P}}
\def\R{\mathbb{R}}
\def\T{\mathbb{T}}
\def\Z{\mathbb{Z}}
\def\ex{\mathrm{ex}}
\renewcommand{\leq}{\leqslant}
\renewcommand{\geq}{\geqslant}
\renewcommand{\to}{\rightarrow}
\newcommand{\pos}{\mathcal{P}}
\renewcommand{\neg}{\mathcal{N}}
\newcommand{\fixed}{\mathcal{Z}}
\newcommand{\free}{\mathcal{F}}
\DeclareMathOperator{\Po}{Po}

\title[The time of b.p. with dense initial sets for all thresholds]{The time of bootstrap percolation with dense initial sets for all thresholds}

\author[B. Bollob\'as]{B\'ela Bollob\'as}
\address{Trinity College, Cambridge, CB2 1TQ, UK, and Department of Mathematical Sciences, University of Memphis, Memphis, TN 38152, USA}
\email{b.bollobas@dpmms.cam.ac.uk}

\author[P.J. Smith]{Paul Smith}
\address{Department of Pure Mathematics and Mathematical Statistics, Wilberforce Road, Cambridge, CB3 0WA, UK}
\email{p.j.smith@dpmms.cam.ac.uk}

\author[A.J. Uzzell]{Andrew J. Uzzell}
\address{Department of Mathematical Sciences, University of Memphis, Memphis, TN 38152, USA}
\email{ajuzzell@memphis.edu}

\date{\today}

\thanks{The first and third authors were partially supported by ARO grant W911NF-06-1-0076 and by NSF grant DMS-0906634. The authors are grateful to Yuval Peres and other members of the Theory Group at Microsoft Research, Redmond, where this research was carried out.}


\begin{document}

\begin{abstract}
We study the percolation time of the $r$-neighbour bootstrap percolation model on the discrete torus $(\Z/n\Z)^d$. For $t$ at most a polylog function of $n$ and initial infection probabilities within certain ranges depending on $t$, we prove that the percolation time of a random subset of the torus is exactly equal to $t$ with high probability as $n$ tends to infinity. Our proof rests crucially on three new extremal theorems that together establish an almost complete understanding of the geometric behaviour of the $r$-neighbour bootstrap process in the dense setting. The special case $d-r=0$ of our result was proved recently by Bollob\'as, Holmgren, Smith and Uzzell.
\end{abstract}

\maketitle

\section{Introduction}

Under \emph{$r$-neighbour bootstrap percolation} on a graph $G$, some of the vertices of $G$ are initially infected, and at each step, infected vertices stay infected, and uninfected vertices become infected if they have at least $r$ infected neighbours. Making this formal, there is a set $A=A_0\subset V(G)$, and for $t\geq 0$,
\[
A_{t+1} = A_t \cup \{v : |\Gamma(v)\cap A_t| \geq r\}.
\]
where $\Gamma(v)$ denotes the set of neighbours of $v$ in the graph $G$. We write $[A]:=\cup_{t=0}^\infty A_t$ for the \emph{closure} of $A$, and say that $A$ \emph{percolates} $G$ (or that \emph{percolation occurs}) if eventually every vertex of $G$ becomes infected; that is, if $[A] = V(G)$. The set $A$ is \emph{closed} if $[A]=A$.

Bootstrap percolation was introduced by Chalupa, Leath and Reich \cite{CLR} as a model for certain interacting particle systems in physics. Since then it has found applications in crack formation, clustering phenomena, dynamics of glasses \cite{GST}, sandpiles \cite{FLP}, the Ising model for ferromagnetism \cite{MorrisGlauber}, jamming \cite{DGLBD}, and many other areas of statistical mechanics and physics, as well as in neural networks \cite{TEneur,Amini}, computer science \cite{DRcomp,FLLPS}, and sociology \cite{Gran,Watts}.

There are two broad classes of questions one can ask about bootstrap percolation. The first, and the most extensively studied, is what happens when the initial configuration $A$ is chosen randomly? Fix a probability $p$ and let $A$ be a random subset of $V(G)$ in which vertices are included independently with probability $p$. One would like to know, for example, how likely percolation is to occur, and if it does occur, how long it takes.

The answer to the first of these questions is now well understood: on the lattice graph $[n]^d$, in which $d$ is fixed and $n$ tends to infinity, the probability of percolation under the $r$-neighbour model displays a sharp threshold between no percolation with high probability and percolation with high probability, meaning that there exists $p_c=p_c(n,d,r)$ such that for all $\epsilon>0$, if $p\geq(1+\epsilon)p_c$ then there is percolation with high probability, while if $p\leq (1-\epsilon)p_c$ then there is no percolation with high probability. The existence of thresholds in a certain weaker sense was proved in papers by Aizenman, Lebowitz, Cerf, Cirillo and Manzo \cite{AL,CerfCir,CerfManzo}, and in the strong sense just described in papers by Holroyd, Balogh, Bollob\'as, Duminil-Copin and Morris \cite{Hol,BBM3D,BBDCM}. Sharp thresholds have also been proved for the hypercube (Balogh and Bollob\'as \cite{BBhyp}, and Balogh, Bollob\'as and Morris \cite{BBMhigh}) and for several other bootstrap models on $\Z^d$ (Duminil-Copin and Holroyd \cite{DCH}, and Duminil-Copin and van Enter \cite{DCvE}).

If $p$ is large enough for percolation to occur with high probability, one would like to know how long percolation takes. In other words, what can one say about the random variable
\[
T = \min \{t \, : \, A_t=V(G) \}?
\]
(If percolation does not occur then we define $T=\infty$.) Asymptotics for $T$ have been obtained by Janson, {\L}uczak, Turova and Vallier \cite{JLTV} on the Erd\H{o}s-R\'enyi random graph $G(n,p)$. In a recent preprint, Balister, Bollob\'as and Smith \cite{BBS} study the question on $[n]^2$ with $r=2$. They prove concentration of $T$ up to a constant factor for all $p$ greater than the critical probability $p_c(n,2,2)$, and asymptotic concentration of $T$ for all $p$ above a larger threshold. Bollob\'as, Holmgren, Smith and Uzzell \cite{BHSU} study the question on the discrete torus $\T_n^d = (\Z/n\Z)^d$ under the $d$-neighbour model. They show that when the initial configuration is dense (in particular requiring $p=1-o(1)$), one can show considerably more than just asymptotic concentration. For $t=t(n)$ up to $o(\log n/\log\log n)$ (including constant $t$) they determine that for certain ranges of the initial probability $p=p(n)$ one can say that $T=t$ with high probability, while for other ranges one obtains $T\in\{t,t+1\}$ with high probability. In this paper we vastly extend these results: we show that there are corresponding one- and two-point concentration theorems for every threshold $2\leq r\leq 2d$. The previous results are the special case $r=d$ of the new results.

Before we state the new results, we need a small amount of notation. Let us write $q_n$ for $1-p_n$ and $q$ for $1-p$. Let
\begin{equation}\label{eq:mt}
m_{d,r}(t) = \sum_{i_0=0}^t \sum_{i_1=0}^{i_0} \dots \sum_{i_{d-r+1}=0}^{i_{d-r}} \binom{d}{i_{d-r+1}}
\end{equation}
for each $1\leq r\leq d$ and arbitrary $t$. This iterated sum, which will appear frequently throughout the paper, is the size of a certain set that is naturally associated with the geometry of the bootstrap process. We shall be more precise about what we mean by this later in the introduction. The following two theorems are our main results.

\begin{theorem}\label{th:main}
Let $d\geq r\geq 2$, let $t = t(n) = o((\log n/\log\log n)^{1/(d-r+1)})$, let $(p_n)_{n=1}^\infty$ be a sequence of probabilities, and let $\omega(n)\to\infty$. Under the standard $r$-neighbour rule on $\T_n^d$,
\begin{enumerate}
\item if $q_n \leq (n^{-d}/\omega(n))^{1/m_{d,r}(t)}$, then $\P_{p_n}(T\leq t)\to 1$ as $n\to\infty$;
\item if $q_n \geq (n^{-d}\omega(n))^{1/m_{d,r}(t)}$, then $\P_{p_n}(T\leq t)\to 0$ as $n\to\infty$.
\end{enumerate}
\end{theorem}

In addition to Theorem \ref{th:main}, we prove that if the sequence $(q_n)_{n=1}^\infty$ satisfies certain bounds, then the percolation time $T$ is determined exactly, or that it takes one of two values, in each case with high probability as $n$ tends to infinity.

\begin{theorem}\label{th:conc}
Let $d\geq r\geq 2$, let $t=o((\log n/\log\log n)^{1/(d-r+1)})$, and let $(p_n)_{n=1}^\infty$ be a sequence of probabilities.
\begin{enumerate}
\item Suppose there exists $\omega(n)\to\infty$ such that
\[
(n^{-d}\omega(n))^{1/m_{d,r}(t-1)} \leq q_n \leq (n^{-d}/\omega(n))^{1/m_{d,r}(t)}.
\]
Then $T=t$ with high probability.
\item Suppose instead there exists a constant $C>0$ such that
\[
(n^{-d}/C)^{1/m_{d,r}(t)} \leq q_n \leq (Cn^{-d})^{1/m_{d,r}(t)}
\]
for all sufficiently large $n$. Then $T\in\{t,t+1\}$ with high probability. If, moreover, there exists a constant $c>0$ such that $q_n^{m_{d,r}(t)}n^d \to c$ as $n\to\infty$, then
\[
\P_{p_n}(T=t) \sim 1-\P_{p_n}(T=t+1) \sim \exp(-g_{d,r}c),
\]
where
\[
g_{d,r} = \binom{d}{d-r+1} 2^{r-1} d^{2(d-r+1)}.
\]
\end{enumerate}
\end{theorem}

Why should one expect such a sharp result as Theorem \ref{th:conc} when $p$ is close to $1$? The key is that with $p$ so large, and therefore the expected percolation time so small, the infection times of almost all pairs of sites in $\T_n^d$ behave independently. This is because if two sites are at $\ell_1$ distance greater than $t$ then their states cannot affect each other by time $t$. This allows one to show with the help of some probabilistic machinery that the number of sites uninfected at time $t$ converges in distribution to a Poisson distribution. We make this statement precise in Theorem \ref{th:dist}. 

When $r>d$, the situation is very different. For $r$ in this range, which is called the \emph{subcritical} range, there exist cofinite sets in $\Z^d$ that are closed under the $r$-neighbour model. This greatly simplifies the analysis of the model, and we shall see in Theorem \ref{th:sub} that it reduces the range of possible percolation times to the finite set $\{0,1,\dots,d,\infty\}$ with high probability.

The second broad class of questions one can ask about bootstrap percolation is the class of extremal questions: for example, what is the minimum or maximum size of $A$ such that a certain property holds, or what is the minimum or maximum time that percolation can take, possibly given certain properties of $A$? The first significant theorems in extremal bootstrap percolation were due to Morris \cite{MorrisMin} and Riedl \cite{Riedl}, who studied the sizes of minimal percolating sets on the square grid and the hypercube $\{0,1\}^d$ respectively. Later, Benevides and Przykucki \cite{BenPrz} determined the maximal percolating time on the square grid, and Przykucki \cite{PrzHyp} did the same for the hypercube.

Our reason for mentioning extremal bootstrap percolation in the introduction to a paper about probabilistic bootstrap percolation is that the proofs of our probabilistic results rest almost entirely on a series of three new extremal theorems that together establish an almost complete analysis of the geometry of the $r$-neighbour bootstrap process in the dense setting. Special cases of these extremal theorems are used in the proofs of the main results in \cite{BHSU}, but the generalizations here are far from straightforward. The key reason for the additional challenges in the general setting is that the geometry of the $r$-neighbour bootstrap process in $\Z^d$ with dense initial sets is essentially $(d-r+1)$-dimensional. Thus, in the special case $r=d$, the geometry of the process is essentially one-dimensional, so it is not surprising that the analysis presents far fewer difficulties.

In order to put the three extremal results into context, first we describe the overall approach of the proofs of Theorems \ref{th:main} and \ref{th:conc}. The Poisson approximation property we described briefly several paragraphs ago is achieved through use of our main probabilistic tool, the Stein-Chen method, which is a means of proving bounds between two probability distributions. Typically one of the distributions is either the normal distribution, as in the original work of Stein \cite{Stein}, or, as here, the Poisson distribution, as later developed by Chen \cite{Chen}. In order to use the Stein-Chen method in our context we need tight control of the first two moments of the number of uninfected sites at time $t$. It is here that the three extremal theorems are needed, and we describe them now.

The first, Theorem \ref{th:min} in the present paper, supposes that a given site $x\in\T_n^d$ is uninfected at time $t$ and ainswers the question: what is the maximum possible size of the set of initially infected sites $A$? A preliminary observation that we have already made is that the states of sites at $\ell_1$ distance greater than $t$ from $x$ cannot affect whether or not $x$ is infected at time $t$, so we may restrict our attention to the $\ell_1$ ball
\[
B_t(x) = \{y\in\T_n^d \, : \, \|x-y\| \leq t\}.
\]
Later we shall need the $\ell_1$ sphere
\[
S_t(x) = \{y\in\T_n^d \, : \, \|y-x\| = t\},
\]
and we write $B_t$ for $B_t(0)$ and $S_t$ for $S_t(0)$. The question now is to determine the quantity
\begin{equation}\label{eq:extrdef}
\ex_{d,r}(t) := \min \{ |B_t \setminus A| \, : \, 0\notin A_t \}.
\end{equation}
If a set $P\subset V(G)$ of vertices has the property that, provided no element of $P$ is in $A$, then no matter what the initial states of the other vertices in $V(G)$ are it can never be the case that $x\in A_t$, then we say that $P$ \emph{protects} $x$ (the time $t$ is implicit). We also say that a site $x$ is \emph{protected} if $x\notin A_{t-\|x\|_1}$. Given $x\in\T_n^d$, let $x_i$ be the $i$th coordinate of $x$ with respect to the standard basis vectors $e_1,\dots,e_d$ in $\R^d$, so $x=(x_1,\dots,x_d)$. A natural example of a set of sites that protects $0$ under the $r$-neighbour process is
\begin{equation}\label{eq:extrset}
P_{d,r}(t) = \{ x\in B_t \, : \, x_{d-r+2},\dots,x_d \in \{0,1\} \},
\end{equation}
which we can think of as the intersection of $B_t(0)$ with a (disjoint) union of $2^{r-1}$ translates of the $(d-r+1)$-dimensional `subspace' $\{x: x_{d-r+2}=\dots=x_d=0\}$ in $\Z^d$, or informally as a $(d-r+1)$-dimensional set in $\Z^d$ with `thickness' $2$. (Of course $\Z^d$ is not a vector space, so it does not make sense to talk about subspaces, but we shall often do so, unambiguously, always meaning the intersection of $\Z^d$ with the corresponding subspace of $\R^d$. The purpose of this slight abuse of nomenclature is to make clear the graph structure of the object under consideration, and in particular the degrees of the sites.) The first extremal theorem (Theorem \ref{th:min}) says that this is best possible: $\ex_{d,r}(t)$ is at least the size of the set $P_{d,r}(t)$ for every $2\leq r\leq d$. Moreover, we verify that the size of $P_{d,r}(t)$ is precisely the quantity $m_{d,r}(t)$ defined in \eqref{eq:mt}, so the full content of the theorem is the statement that
\begin{equation}\label{eq:min}
\ex_{d,r}(t) = |P_{d,r}(t)| = m_{d,r}(t).
\end{equation}

The second of the three extremal theorems, Theorem \ref{th:extremal}, says that sets of the form of $P_{d,r}(t)$ are essentially the only extremal sets. By `essentially' here we mean that there do exist other extremal sets, but there are a constant number of them, and apart from rotations and reflections, they only differ from $P_{d,r}(t)$ in the positions of at most $2(d-r+1)$ sites. The constant number of them that there are is the quantity $g_{d,r}$ defined in the statement of Theorem \ref{th:conc}.

Now suppose instead that the number of initially infected sites is not the minimal number, but is close to the minimum. What can we say about the positions of these sites now? The third extremal theorem, Theorem \ref{th:stability}, is a stability theorem that says that under these conditions the sites must look a lot like a set of the form of $P_{d,r}(t)$, in a certain specific sense. As well as being interesting in its own right, this theorem immediately implies that the number of sets of size $m_{d,r}(t)+a$ uninfected sites in $B_t$ that protect the origin is $t^{O(a)}$, which is much smaller than the trivial bound of $t^{O(ta)}$ when $a$ is small.

As we mentioned earlier, the proofs of the three extremal theorems present significant new difficulties over the results of \cite{BHSU} because the geometry of the bootstrap process in the dense setting is no longer one-dimensional, but $(d-r+1)$-dimensional. (The result of Theorem \ref{th:extremal}, that sets like $P_{d,r}(t)$ are essentially the only extremal sets, makes this claim a little more substantive.) We now expand on the key difficulties in generalizing the results of \cite{BHSU} to other thresholds $r$ and some of the ideas we use to overcome those difficulties. 

The proof of the first extremal theorem in \cite{BHSU} proceeds by induction on two variables: the time parameter $t$ and the number $f$ of so-called free coordinates (see Section \ref{se:min} for a definition). Owing to the complications introduced by the higher-dimensional geometry of the process in the general $r$ setting, the proof of the first extremal theorem in the present paper proceeds by induction on three variables; the new variable is the difference $d-r$. During the course of this induction we have to keep track of a large collection of disjoint subsets of $B_t$ of different dimensions (Figure \ref{fi:case2} illustrates the sets in the case $d=3$, $r=2$). This is achieved through the use of several new ideas, including compatibility functions and restrictions, which are defined early in Section \ref{se:min}.

For the second extremal theorem, the simplest of the three, the proof in \cite{BHSU} splits into three cases according to the relative sizes of $t$ and $d$. That proof would not have generalized to a bounded number of cases for general thresholds, so here we present a new, considerably streamlined proof.

The proof of the stability theorem in \cite{BHSU} uses a lemma (Lemma $3.7$ in that paper) about connected components of uninfected sites. The corresponding statement for thresholds other than $r=d$ is false, which has meant that we have had to introduce a completely new approach to the proof of the stability theorem in this paper. This new approach makes several uses of a technique that allows one to show the following. Suppose the origin is protected and, for some $k$, the ball $B_k$ contains the minimum number of protected sites $\ex_{d,r}(k)$. (Actually the assumption is allowed to be much weaker than that, but we do not go into the details here.) Then provided the sites $-e_1$ and $e_1$ are both protected, one can show that the intersection of $B_k$ with the hyperplane $\{x\in\Z^d:x_1=0\}$ cannot contain too many protected sites: in fact, they must also be minimal. The power of this technique, which we call the \emph{hyperplane restriction principle}, is that it allows one to pass to a lower dimensional space where we may apply induction. An extended sketch of the proof of Theorem \ref{th:stability} is given in Section \ref{se:stab}.

There are also versions of our results for the \emph{modified} $r$-neighbour bootstrap percolation model. In this model, on the graph $\T_n^d$, there is again an initial set $A$ of infected sites, and for $t\geq 0$ we set
\[
A_{t+1} = A_t \cup \{v : |\{v+e_i,v-e_i\} \cap A_t| \geq 1 \text{ for at least $r$ distinct choices of $i\in[d]$}\}.
\]

With some slight simplifications, our arguments can be used to prove the following result.

\begin{theorem}\label{th:mod}
Under the modified $r$-neighbour rule on $\T_n^d$, Theorems \ref{th:main} and \ref{th:conc} hold, with $m_{d,r}(t)$ replaced by $m_{d,r}'(t)$ and $g_{d,r}$ replaced by $g_{d,r}'$, where
\[
m_{d,r}'(t) = \sum_{i_0=0}^{d-r+1} \binom{d-r+1}{i_0} \sum_{i_1=0}^{t-i_0} \sum_{i_2=0}^{i_1} \dots \sum_{i_{d-r+1}=0}^{i_{d-r}} 1
\]
is the volume of a $(d-r+1)$-dimensional $\ell_1$ ball of radius $t$, and
\[
g_{d,r}' = \binom{d}{d-r+1}.
\]
\end{theorem}

The rest of this paper is organized as follows. In Section \ref{se:comb} we prove two basic results about binomial coefficients, which are needed in the proofs of our main extremal results. In Sections \ref{se:min} and \ref{se:stab} we study minimal and near-minimal protecting sets respectively, and prove the three extremal theorems. We bring these results together in Section \ref{se:main}, and with the help of some standard probabilistic tools, use them to prove Theorems \ref{th:main} and \ref{th:conc}. Finally, in Section \ref{se:sub}, we prove corresponding results for subcritical models.

\section{Combinatorial preliminaries}\label{se:comb}

The purpose of this section is to prove two easy identities concerning sums of binomial coefficients; we shall use these repeatedly throughout the next two sections.

\begin{lemma}\label{le:minibinom}
Let $d\geq r\geq 2$, let $f\geq 0$, and let $k\geq 0$. Then
\begin{multline}\label{eq:minibinom}
\sum_{i_1=0}^k \sum_{i_2=0}^{i_1} \cdots \sum_{i_{d-r+1}=0}^{i_{d-r}} \binom{f}{i_{d-r+1}} \\
= 2\sum_{i_1=0}^{k-1} \sum_{i_2=0}^{i_1} \cdots \sum_{i_{d-r+1}=0}^{i_{d-r}} \binom{f-1}{i_{d-r+1}} + \sum_{i_2=0}^k \sum_{i_3=0}^{i_2} \cdots \sum_{i_{d-r+1}=0}^{i_{d-r}} \binom{f-1}{i_{d-r+1}}.
\end{multline}
\end{lemma}

In Section \ref{se:min} we show that the left-hand side of \eqref{eq:minibinom} is the volume of the surface of an extremal set of radius $k$, as in \eqref{eq:extrset}. The lemma can be thought of as saying that this volume is equal to the volume of that part of the surface that lies in a codimension $1$ subspace (the second term on the right-hand side of \eqref{eq:minibinom}) plus the volume of the surface that lies in each hyperplane parallel to (but distinct from) the subspace.

\begin{proof}
We shall prove the identity by induction on $d-r$. When $d-r=0$, \eqref{eq:minibinom} is equivalent to
\[
\sum_{i_1=0}^k \binom{f}{i_1} = 2\sum_{i_1=0}^{k-1} \binom{f-1}{i_1} + \binom{f-1}{k}.
\]
Rewriting the right-hand side as
\[
\binom{f-1}{0} + \left(\binom{f-1}{0}+\binom{f-1}{1}\right) + \dots + \left(\binom{f-1}{k-1}+\binom{f-1}{k}\right),
\]
we see that the identity holds.

Suppose the lemma holds for $d-r-1$. After re-indexing the second expression, the right-hand side of \eqref{eq:minibinom} is equal to
\begin{equation}\label{eq:microbinom}
2\sum_{i_1=0}^{k-1} \sum_{i_2=0}^{i_1} \cdots \sum_{i_{d-r+1}=0}^{i_{d-r}} \binom{f-1}{i_{d-r+1}} + \sum_{i_1=0}^k \sum_{i_2=0}^{i_1} \cdots \sum_{i_{d-r}=0}^{i_{d-r-1}} \binom{f-1}{i_{d-r}}.
\end{equation}
We can rewrite the first expression here as
\begin{align}\label{eq:nanobinom1}
2&\sum_{i_1=0}^{k-1} \left( \sum_{i_2=0}^{i_1-1} \cdots \sum_{i_{d-r+1}=0}^{i_{d-r}} \binom{f-1}{i_{d-r+1}} + \sum_{i_3=0}^{i_1} \cdots \sum_{i_{d-r+1}=0}^{i_{d-r}} \binom{f-1}{i_{d-r+1}} \right) \notag \\
=&\sum_{i_1=0}^{k-1} 2 \sum_{i_2=0}^{i_1-1} \cdots \sum_{i_{d-r+1}=0}^{i_{d-r}} \binom{f-1}{i_{d-r+1}} + 2\sum_{i_1=0}^{k-1} \sum_{i_2=0}^{i_1} \cdots \sum_{i_{d-r}=0}^{i_{d-r-1}} \binom{f-1}{i_{d-r}},
\end{align}
where in the second line we have just moved the factor of $2$ inside one sum and re-indexed the second sum. We can also rewrite the second expression in \eqref{eq:microbinom} as
\begin{equation}\label{eq:nanobinom2}
\sum_{i_2=0}^k \sum_{i_3=0}^{i_2} \cdots \sum_{i_{d-r}=0}^{i_{d-r-1}} \binom{f-1}{i_{d-r}} + \sum_{i_1=0}^{k-1} \sum_{i_2=0}^{i_1} \cdots \sum_{i_{d-r}=0}^{i_{d-r-1}} \binom{f-1}{i_{d-r}}.
\end{equation}
Summing \eqref{eq:nanobinom1} and \eqref{eq:nanobinom2} we obtain that \eqref{eq:microbinom} is equal to
\begin{multline*}
\sum_{i_1=0}^{k-1} \left( 2\sum_{i_2=0}^{i_1-1} \cdots \sum_{i_{d-r+1}=0}^{i_{d-r}} \binom{f-1}{i_{d-r+1}} + \sum_{i_2=0}^{i_1} \cdots \sum_{i_{d-r}=0}^{i_{d-r-1}} \binom{f-1}{i_{d-r}} \right) \\
+ \left( 2\sum_{i_1=0}^{k-1} \sum_{i_2=0}^{i_1} \cdots \sum_{i_{d-r}=0}^{i_{d-r-1}} \binom{f-1}{i_{d-r}} + \sum_{i_2=0}^k \sum_{i_3=0}^{i_2} \cdots \sum_{i_{d-r}=0}^{i_{d-r-1}} \binom{f-1}{i_{d-r}} \right).
\end{multline*}
Applying induction twice, once to each of the expressions inside the two sets of large brackets, we find this is equal to
\[
\sum_{i_1=0}^{k-1} \sum_{i_2=0}^{i_1} \cdots \sum_{i_{d-r+1}=0}^{i_{d-r}} \binom{f}{i_{d-r+1}} + \sum_{i_1=0}^k \sum_{i_2=0}^{i_1} \cdots \sum_{i_{d-r}=0}^{i_{d-r-1}} \binom{f}{i_{d-r}},
\]
which, upon re-indexing the second expression one last time and combining the sums, proves the lemma.
\end{proof}

The next lemma is an iterated version of the previous one, and has a similar iterated interpretation.

\begin{lemma}\label{le:megabinom}
Let $d>r\geq 2$, let $0\leq f\leq d$, and let $k\geq 0$. Then
\begin{align}\label{eq:megabinom}
\sum_{i_1=0}^k \sum_{i_2=0}^{i_1} \cdots \sum_{i_{d-r+1}=0}^{i_{d-r}} \binom{f}{i_{d-r+1}} = \; &2\sum_{i_1=0}^{k-1} \sum_{i_2=0}^{i_1} \cdots \sum_{i_{d-r+1}=0}^{i_{d-r}} \binom{f-1}{i_{d-r+1}} \notag \\
&+ 2\sum_{i_2=0}^{k-1} \sum_{i_3=0}^{i_1} \cdots \sum_{i_{d-r+1}=0}^{i_{d-r}} \binom{f-2}{i_{d-r+1}} + \dots \notag \\
&+ 2\sum_{i_{d-r+1}=0}^{k-1} \binom{f-d+r-1}{i_{d-r+1}} + \binom{f-d+r-1}{k}.
\end{align}
\end{lemma}

\begin{proof}
Again, we shall prove the identity by induction on $d-r$. When $d-r=1$, the claim is precisely the same as Lemma \ref{le:minibinom}.

Suppose the lemma holds for $d-r-1$. It follows that the right-hand side of \eqref{eq:megabinom} is equal to
\[
2\sum_{i_1=0}^{k-1} \sum_{i_2=0}^{i_1} \cdots \sum_{i_{d-r+1}=0}^{i_{d-r}} \binom{f-1}{i_{d-r+1}} + \sum_{i_2=0}^k \sum_{i_3=0}^{i_2} \cdots \sum_{i_{d-r+1}=0}^{i_{d-r}} \binom{f-1}{i_{d-r+1}}.
\]
This is equal to the left-hand side of \eqref{eq:megabinom} by Lemma \ref{le:minibinom}.
\end{proof}

\section{Minimal configurations}\label{se:min}

In this section we prove two extremal theorems about sets of uninfected sites with certain properties. The first theorem determines the extremal number of sites defined in \eqref{eq:extrdef} and shows that it is equal to the size of the natural example of a protecting set defined in \eqref{eq:extrset}. The second says that sets like the one in \eqref{eq:extrset} are essentially the only extremal sets.

We need quite a lot of notation before we can state the results of this section in the level of generality that we need later in the paper.

First, define a partial order on sites in $\Z^d$ as follows. For $x,y\in\Z^d$, we say that $x\leq y$, or that $y\geq x$, or that $y$ is \emph{above} $x$, if $y_i\geq x_i$ for each $i$ such that $x_i>0$, and $y_i\leq x_i$ for each $i$ such that $x_i<0$.

Given a graph $G$, a vertex $x\in V(G)$, and a subgraph $H$ of $G$, we use the standard graph-theoretic notation $d_H(x)$ to mean the degree of $x$ in $H$. Thus, $d_H(x):=\big|\big\{y\in V(H):\{x,y\}\in E(G)\big\}\big|$.

Throughout this section and the next, in which we focus on extremal rather than probabilistic questions, the time $t$ will be fixed, and because it is fixed we shall only rarely mention it explicitly. Recall that a site $x\in\T_n^d$ is said to be \emph{protected} if $x\notin A_{t-\|x\|_1}$. This is a natural definition: one should think of it as saying that $x$ is protected if it is still uninfected at the last time its state could affect the state of the origin at time $t$. Clearly this a stronger statement than saying that $x$ is initially uninfected. We also write $P(X)$ for the set of protected sites in a set $X$.

Two neighbours $y_1$ and $y_2$ of a site $x$ are said to be \emph{opposing} if there exists $i$ such that $y_1=x+e_i$ and $y_2=x-e_i$, or vice-versa.

The next few definitions allow us to be more specific about where we are looking for sites protecting a given site. The extra control we gain over the positions of these sites is not needed in this section, but it will be necessary for the stability result in the next section.

A function $C:[d]\to\{-1,0,1,\ast\}$ is called a \emph{compatibility function}. A site $y$ is \emph{$C$-compatible} with a site $x$ if the following three conditions hold:
\begin{enumerate}
\item $y_i-x_i \geq 0$ if $C(i)=1$;
\item $y_i-x_i \leq 0$ if $C(i)=-1$;
\item $y_i=x_i$ if $C(i)=0$.
\end{enumerate}
If $C(i)=\ast$ then there is no restriction on $y_i$. Given a compatibility function $C$, we write $\pos(C)$ for the set $\{i:C(i)=1\}$ of \emph{positive coordinates}, $\neg(C)$ for the set $\{i:C(i)=-1\}$ of \emph{negative coordinates}, $\fixed(C)$ for the set $\{i:C(i)=0\}$ of \emph{fixed coordinates}, and $\free(C)$ for the set $\{i:C(i)=\ast\}$ of \emph{free coordinates}. For a site $x$, an integer $k\geq 0$ and a compatibility function $C$, let $P_k^C(x)$ denote the set of all protected sites in $S_k(x)$ that are $C$-compatible with $x$. The \emph{$i$-restriction} of a compatibility function $C$ is the compatibility function $C'$ that satisfies $C'(i)=0$ and $C'(j)=C(j)$ for all $j\neq i$.

The following lemma for subcritical $(d+1)$-neighbour bootstrap percolation is proved in \cite{BHSU}, although it is not explicitly stated, and it is also a special case of Lemma \ref{le:sub} in this paper.

\begin{lemma}\label{le:d+1}
Let $k\in\N$ and let $d\geq 2$. Suppose that $x \in B_k$ is protected under $(d+1)$-neighbour bootstrap percolation. Let $C$ be a compatibility function with no fixed coordinates and let $f = |\free(C)|$ be the number of free coordinates of $C$. Then
\begin{equation}
|P_k^C(x)| \geq \binom{f}{k}. \tag*{\qed}
\end{equation}
\end{lemma}

We shall use Lemma \ref{le:d+1} in the proof of the next lemma, which will be the key lemma in the proof of the first extremal theorem.

\begin{lemma}\label{le:key}
Let $k \in \N$ and let $2\leq r\leq d$. Suppose that $x \in B_k$ is protected under $r$-neighbour bootstrap percolation. Let $C$ be a compatibility function with no fixed coordinates and let $f = |\free(C)|$ be the number of free coordinates of $C$. Then
\[
|P_k^C(x)| \geq \sum_{i_1=0}^k \sum_{i_2=0}^{i_1} \cdots \sum_{i_{d-r+1}=0}^{i_{d-r}} \binom{f}{i_{d-r+1}}.
\]
\end{lemma}

In the applications of Lemma \ref{le:key} in this section (but not in the next) we shall always take $\pos(C)=\{i:x_i>0\}$, $\neg(C)=\{i:x_i<0\}$, $\free(C)=\{i:x_i=0\}$, and $\fixed(C)=\emptyset$. It may be helpful to have this in mind during the proof.

Here is an outline of the proof. Without loss of generality we assume that $\pos(C)=[d-f]=[d]\setminus\free(C)$. The argument runs by induction on the sum $(d-r)+k+f$. The base cases and the case $r=d$ are proved in \cite{BHSU}. There are two possibilities for the induction step. First, suppose that $x$ has at least $d-r+1$ pairs of opposing protected neighbours. We may assume that they are $x\pm e_{d-f+1},\dots,x\pm e_{2d-f-r+1}$. Divide up the sites above $x$ in $B_k$ as follows (thinking of $C$ as in the previous paragraph, and `above' now as being relative to the partial order imposed on sites by $C$). First, take the sites above either $x+e_{d-f+1}$ or $x-e_{d-f+1}$. These sets look like half spaces, in the sense that the only restriction is an inequality on $(d-f+1)$th coordinate. Next, take the sites not in either of those sets (so their $(d-f+1)$th coordinate is equal to that of $x$), but which are above either $x+e_{d-f+2}$ or $x-e_{d-f+2}$. These sets look like half-spaces inside hyperplanes, because of the restriction that the $(d-f+1)$th coordinate is equal to that of $x$ and that the next coordinate must satisfy an inequality. We continue, looking at smaller sets each time. Inside each of these sets, which are disjoint, we apply the induction hypothesis separately, and we use Lemma \ref{le:megabinom} to show that we have found the right number of protected sites. This completes the case when $x$ has lots of opposing protected neighbours. Now suppose $x$ has at most $d-r$ pairs of opposing protected neighbours. In this case, which is the easier of the two, $x$ must have a protected neighbour in a direction among the first $d-f$ coordinates, which we may assume is $x+e_1$. By induction, there are lots of protected sites above $x+e_1$, and also by induction we can find lots of protected sites above $x$ but with first coordinate equal to that of $x$. It turns out that the total number of protected sites we get is the right number.

\begin{proof}[Proof of Lemma \ref{le:key}]
Without loss of generality, let $\pos(C)=[d]\setminus\free(C)=[d-f]$. Let $x=(x_1,\dots,x_d)$ and suppose that $x_i\geq 0$ for all $i\in[d]$.

The proof is by induction on $q:=(d-r)+f+k$. If $q=0$ then we must have $k=0$ and the claim is simply that $x$ itself is protected, which is trivial. The lemma is proved for all $k$ and $f$ when $r=d$ in \cite{BHSU}. We divide the remainder of the proof into two cases, according to whether $x$ has many or few pairs of opposing $C$-compatible protected neighbours.

\begin{figure}[ht]
  \centering
  \begin{tikzpicture}[scale=0.5,>=latex]
    \fill[gray!30] (30:11) ++(0,1) -- ++(-20:1) -- ++(0,-1) -- ++($ (210:10) +(-20:8) $) -- ++(210:1) -- ++($ (160:8) +(210:10) $) -- ++(160:1) -- ++(0,1) -- ++($ (30:10) +(160:8) $) -- ++(30:1) -- ++($ (-20:8) +(30:10) $);
    \fill[gray!90] (0,0) -- ++(-20:1) -- ++(30:1) -- ++(0,1) -- ++(160:1) -- ++(210:1) -- ++(0,-1);

    \draw (0,8) -- ++(-20:1) -- ++(30:1) -- ++(160:1) -- ++(210:1);
    \draw[densely dashed] (0,0) -- ++(-20:1) -- ++(30:1) -- ++(160:1) -- ++(210:1) -- ++(0,1) -- ++(-20:1) -- ++(30:1) -- ++(160:1) -- ++(210:1) (-20:1) -- ++(0,1) ++(30:1) -- ++(0,-1) ++(160:1) -- ++(0,1);
    \draw[densely dashed] (210:10) -- ++($ (30:10) +(160:8) $) -- ++(0,1) ++(0,-1) -- ++(30:1) -- ++(0,1) ++(0,-1) -- ++($ (-20:8) +(30:10) $) -- ++(0,1) ++(0,-1) -- ++(-20:1) ++(160:1) ++(0,1) -- ++($ (210:10) +(160:8) $);
    \draw ++(210:10) -- ++(0,1) -- ++($ (30:10) +(0,7) $) -- ++(30:1) -- ++($ (0,-7) +(30:10) $);
    \draw (-20:1) ++(210:10) -- ++(0,1) -- ++($ (30:10) +(0,7) $) -- ++(30:1) -- ++($ (0,-7) +(30:10) $) -- ++(0,-1);
    \draw (30:11) ++(0,1) -- ++(-20:1) -- ++($ (210:10) +(-20:8) $) -- ++(210:1) -- ++($ (160:8) +(210:10) $) -- ++(160:1) -- ++($ (30:10) +(160:8) $);
    \draw (30:11) ++(-20:1) -- ++($ (210:10) +(-20:8) $) -- ++(210:1) -- ++($ (160:8) +(210:10) $) -- ++(160:1);
    \draw (160:8) ++(0,1) -- ++($ (-20:8) +(0,7) $) ++(-20:1) -- ++($ (0,-7) +(-20:8) $) -- ++(0,-1) ++(30:1) -- ++(0,1) -- ++($ (160:8) +(0,7) $) ++(160:1) ++($ (0,-8) +(160:8) $);
  \end{tikzpicture}
  \caption{Both this figure and Figure \ref{fi:case2} depict the set $B:=\{(x,y,z)\in\Z^3: z\geq 0, \, |x|+|y|+|z|\leq k\}$ with the origin represented by the dark cell at the centre. Letting $C$ be the compatibility function $C=(+,\ast,\ast)$, the set $B$ may also be viewed as the intersection of the $\ell_1$ ball $B_t(0)$ with the set of sites that are $C$-compatible with $0$. Suppose $0$ is protected under the $2$-neighbour model. In this example, Case $1$ of the proof of Lemma \ref{le:key} corresponds to only three of the four neighbours of $0$ in the hyperplane $H:=\{(x,y,0)\in\Z^3\}$ (shown here as the shaded area) being protected. This implies that $e_3$, the site immediately above $0$, must be protected. The proof proceeds by showing inductively that there must be many protected sites above $e_3$, and just enough additional protected sites inside $H$. Note that these two sets of protected sites are clearly disjoint.}
  \label{fi:case1}
\end{figure}
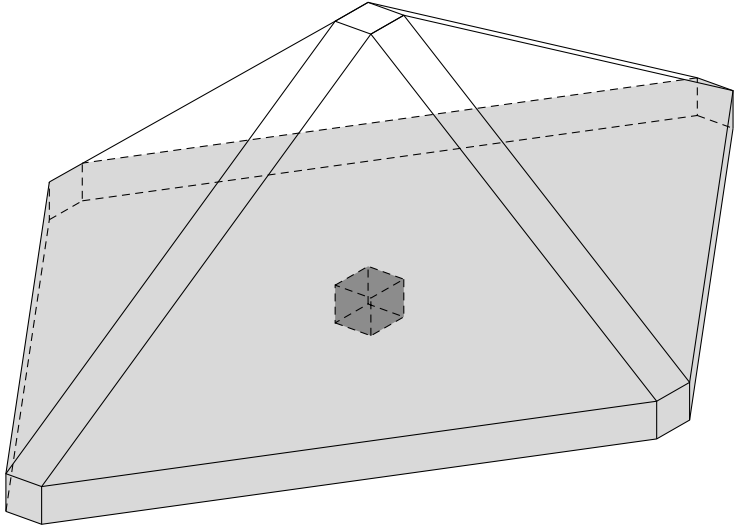

\emph{Case 1: $x$ has at most $d-r$ pairs of opposing $C$-compatible protected neighbours.} The site $x$ is protected, so it must have at least $2d-r+1$ protected neighbours. Exactly $d-f$ neighbours of $x$ are not $C$-compatible, so $x$ must have at least $d+f-r+1$ $C$-compatible protected neighbours. Now, $C$ has $f$ free coordinates, and $x$ has at most $d-r$ pairs of opposing $C$-compatible protected neighbours, so it has at most $d+f-r$ protected neighbours of the form $x+e_i$ or $x-e_i$, where $i\in\{d-f+1,\dots,d\}$. It follows that $x$ has at least one protected neighbour of the form $x+e_i$, for some $i\in[d-f]$. Let us assume that $x':=x+e_1$ is protected. Observe that $P_{k-1}^C(x') \subset P_k^C(x)$, since $x'$ differs from $x$ only in one of its positive coordinates. By induction,
\begin{equation}\label{eq:ineq1}
|P_{k-1}^C(x')| \geq \sum_{i_1=0}^{k-1} \sum_{i_2=0}^{i_1} \cdots \sum_{i_{d-r+1}=0}^{i_{d-r}} \binom{f}{i_{d-r+1}}.
\end{equation}

Let $C'$ be the $1$-restriction of $C$. The remaining sites that we need to complete this case of the proof will be in $P_k^{C'}(x)$. A three-dimensional example is depicted in Figure \ref{fi:case1}. Observe that
\begin{equation}\label{eq:keydisjoint}
P_{k-1}^C(x') \cap P_k^{C'}(x) = \emptyset,
\end{equation}
because sites in $P_{k-1}^C(x')$ have first coordinate at least $x_1+1$ and sites in $P_k^{C'}(x)$ have first coordinate exactly $x_1$. Let $U$ be the codimension $1$ subspace given by
\[
U = \{y\in\Z^d:y_1=x_1\}.
\]
The set of sites which are $C'$-compatible with $x$ in $\Z^d$ is the same as the set of sites which are $C$-compatible with $x$ in $U$. Thus, by induction there are at least
\begin{equation}\label{eq:ineq2}
\sum_{i_2=0}^k \sum_{i_3=0}^{i_2} \cdots \sum_{i_{d-r+1}=0}^{i_{d-r}} \binom{f}{i_{d-r+1}}
\end{equation}
$C$-compatible protected sites at distance $k$ from $x$ all contained in $U$. By \eqref{eq:ineq1}, \eqref{eq:keydisjoint} and \eqref{eq:ineq2}, the total number number of $C$-compatible protected sites at distance $k$ from $x$ is
\begin{align*}
|P_k^C(x)| &\geq \sum_{i_1=0}^{k-1} \sum_{i_2=0}^{i_1} \cdots \sum_{i_{d-r+1}=0}^{i_{d-r}} \binom{f}{i_{d-r+1}} + \sum_{i_2=0}^k \sum_{i_3=0}^{i_2} \cdots \sum_{i_{d-r+1}=0}^{i_{d-r}} \binom{f}{i_{d-r+1}} \\
&= \sum_{i_1=0}^k \sum_{i_2=0}^{i_1} \cdots \sum_{i_{d-r+1}=0}^{i_{d-r}} \binom{f}{i_{d-r+1}},
\end{align*}
as claimed.

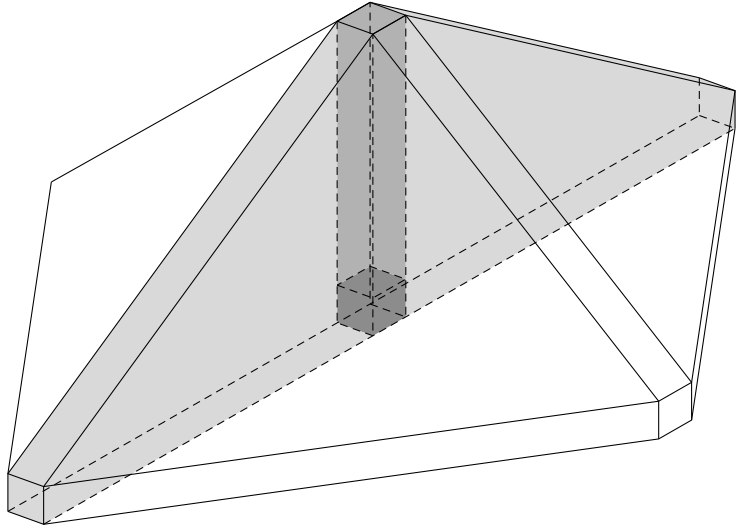
\begin{figure}[ht]
  \centering
  \begin{tikzpicture}[scale=0.5,>=latex]
    \fill[gray!90] (0,0) -- ++(-20:1) -- ++(30:1) -- ++(0,1) -- ++(160:1) -- ++(210:1) -- ++(0,-1);
    \fill[gray!60] (0,1) -- ++(30:1) -- ++(-20:1) -- ++(0,7) -- ++(160:1) -- ++(210:1) -- ++(0,-7);
    \fill[gray!30] (0,0) -- ++(0,8) -- ++($ (0,-7) +(210:10) $) -- ++(0,-1) -- ++(-20:1) -- ++ (30:10) -- ++(160:1);
    \fill[gray!30] (-20:1) ++(30:1) -- ++(30:10) -- ++(0,1) -- ++(160:1) -- ++($ (210:10) +(0,7) $) -- ++(-20:1) -- ++(0,-8);
    \draw[densely dashed] (0,8) -- ++(0,-8) -- ++(-20:1) -- ++(0,8) ++(30:1) -- ++(0,-8) -- ++(160:1) -- ++(0,8);
    \draw (0,8) -- ++(-20:1) -- ++(30:1) -- ++(160:1) -- ++(210:1);
    \draw[densely dashed] (0,1) -- ++(-20:1) -- ++(30:1) -- ++(160:1) -- ++(210:1);
    \draw[densely dashed] (210:10) -- ++(30:21) -- ++(0,1) (-20:1) ++(210:10) -- ++(30:21) -- ++(160:1);
    \draw ++(210:10) -- ++(0,1) -- ++($ (30:10) +(0,7) $) -- ++(30:1) -- ++($ (0,-7) +(30:10) $);
    \draw (-20:1) ++(210:10) -- ++(0,1) -- ++($ (30:10) +(0,7) $) -- ++(30:1) -- ++($ (0,-7) +(30:10) $) -- ++(0,-1);
    \draw (30:11) ++(0,1) -- ++(-20:1) -- ++($ (210:10) +(-20:8) $) -- ++(210:1) -- ++($ (160:8) +(210:10) $) -- ++(160:1) -- ++($ (30:10) +(160:8) $);
    \draw (30:11) ++(-20:1) -- ++($ (210:10) +(-20:8) $) -- ++(210:1) -- ++($ (160:8) +(210:10) $) -- ++(160:1);
    \draw (160:8) ++(0,1) -- ++($ (-20:8) +(0,7) $) ++(-20:1) -- ++($ (0,-7) +(-20:8) $) -- ++(0,-1) ++(30:1) -- ++(0,1) -- ++($ (160:8) +(0,7) $) ++(160:1) ++($ (0,-8) +(160:8) $);
  \end{tikzpicture}
  \caption{In this figure, which has the same setup as Figure \ref{fi:case1}, Case $2$ of the proof of Lemma \ref{le:key} corresponds to the situation in which all four neighbours of the origin that intersect the hyperplane $H$ are protected. The sites that are $C_1^+$-compatible with $e_1$ or $C_1^-$ compatible with $-e_1$ are shown in the unshaded areas; the sites that are $C_2^+$-compatible with $e_2$ or $C_2^-$ compatible with $-e_2$ are shown in the lightly shaded triangles; and the sites that are $C_3$-compatible with $e_3$ are shown in the medium shaded column. Each of these fives disjoint sets must contain many protected sites.}
  \label{fi:case2}
\end{figure}

\emph{Case 2: $x$ has at least $d-r+1$ pairs of opposing $C$-compatible protected neighbours.} Without loss of generality the opposing protected neighbours are $x\pm e_{d-f+1},\dots,x\pm e_{2d-f-r+1}$. (We have $2d-f-r+1\leq d$, or equivalently $f\geq d-r+1$, because every coordinate in which $x$ has a pair of $C$-compatible neighbours must necessarily be free.) Let $C_1=C$ and for $i=1,\dots,d-r+1$ let $C_{i+1}$ be the $(d-f+i)$-restriction of $C_i$. Also, for $i=1,\dots,d-r+1$, let $C_i^+$ be the compatibility function satisfying $C_i^+(d-f+i)=1$ and $C_i^+(j)=C_i(j)$ for $j\neq i$, and let $C_i^-$ be the compatibility function satisfying $C_i^-(d-f+i)=-1$ and $C_i^-(j)=C_i(j)$ for $j\neq i$. The key observation here is that the $2d-2r+3$ sets
\begin{gather*}
P_{k-1}^{C_1^+}(x+e_{d-f+1}) \text{ and } P_{k-1}^{C_1^-}(x-e_{d-f+1}), \\
P_{k-1}^{C_2^+}(x+e_{d-f+2}) \text{ and } P_{k-1}^{C_2^-}(x-e_{d-f+2}), \\
\vdots \\
P_{k-1}^{C_{d-r+1}^+}(x+e_{2d-f-r+1}) \text{ and } P_{k-1}^{C_{d-r+1}^-}(x-e_{2d-f-r+1}), \\
\text{and } P_{k-1}^{C_{d-r+2}}(x)
\end{gather*}
 are all pairwise disjoint, so we can obtain a bound on $P_k^C(x)$ by bounding the sizes of each of these sets individually. The reason for their disjointness is as follows. First, for each $i$, $P_{k-1}^{C_i^+}(x+e_{d-f+i})$ and $P_{k-1}^{C_i^-}(x-e_{d-f+i})$ are easily seen to be disjoint because the $(d-f+i)$th coordinates of sites in the first set are all greater than the $(d-f+i)$th coordinate of $x$, and the $(d-f+i)$th coordinates of sites in the second set are all less than the $(d-f+i)$th coordinate of $x$. Furthermore, the sets that appear after $P_{k-1}^{C_i^+}(x+e_{d-f+i})$ and $P_{k-1}^{C_i^-}(x-e_{d-f+i})$ in the list, which are
\begin{gather*}
P_{k-1}^{C_{i+1}^+}(x+e_{d-f+i+1}) \text{ and } P_{k-1}^{C_{i+1}^-}(x-e_{d-f+i+1}), \\
\vdots \\
P_{k-1}^{C_{d-r+1}^+}(x+e_{2d-f-r+1}) \text{ and } P_{k-1}^{C_{d-r+1}^-}(x-e_{2d-f-r+1}), \\
\text{and } P_{k-1}^{C_{d-r+2}}(x)
\end{gather*}
only contain sites that have $(d-f+i)$th coordinate equal to that of $x$, by the definition of restrictions. Therefore these sets are all also disjoint from both $P_{k-1}^{C_i^+}(x+e_{d-f+i})$ and $P_{k-1}^{C_i^-}(x-e_{d-f+i})$. This proves the disjointness property. Figure \ref{fi:case2} illustrates an example of these disjoint sets in three dimensions.

The compatibility function $C_i$ has exactly $i-1$ fixed coordinates. In fact, all sites that are $C_i$-compatible with $x$ lie inside the $(d-i+1)$-dimensional affine subspace
\[
U_i = \{y\in\Z^d : y_{d-f+1}=x_{d-f+1}, \dots, y_{d-f+i-1}=x_{d-f+i-1}\}.
\]
(We define $U_1=\Z^d$.) Thus, when we are looking for $C_i$-compatible protected sites, we are really looking for protected sites inside a $(d-i+1)$-dimensional space. The function $C_i$ has $f-i+1$ free coordinates, and $C_i^+$ and $C_i^-$ each have $f-i$ free coordinates. By induction, for each $i=1,\dots,d-r+1$ we have
\begin{equation}\label{eq:protCiplus}
|P_{k-1}^{C_i^+}(x+e_{d-f+i})| \geq \sum_{i_1=0}^{k-1} \sum_{i_2=0}^{i_1} \cdots \sum_{i_{d-i-r+2}=0}^{i_{d-i-r+1}} \binom{f-i}{i_{d-i-r+2}},
\end{equation}
and a similar inequality holds for $|P_{k-1}^{C_i^-}(x-e_{d-f+i})|$. Furthermore, $C_{d-r+2}$ has $f-d+r-1$ free coordinates, and sites which are $C_{d-r+2}$-compatible with $x$ all lie in the $(r-1)$-dimensional affine subspace $U_{d-r+2}$. By Lemma \ref{le:d+1},
\begin{equation}\label{eq:protbase}
|P_{k-1}^{C_{d-r+2}}(x)| \geq \binom{f-d+r-1}{k}.
\end{equation}
Summing \eqref{eq:protCiplus} over $i=1,\dots,d-r+1$ and each choice of $+$ or $-$, and adding \eqref{eq:protbase} to the sum, we obtain precisely the right-hand side of the identity \eqref{eq:megabinom}. Lemma \ref{le:megabinom} then completes this case of the proof.
\end{proof}

\begin{corollary}\label{co:min}
Let $t\in\N$ and let $2\leq r\leq d$. Suppose that the origin is protected under $r$-neighbour bootstrap percolation. Then for $k=0,\dots,t$,
\[
|P(S_k)| \geq \sum_{i_1=0}^k \sum_{i_2=0}^{i_1} \cdots \sum_{i_{d-r+1}=0}^{i_{d-r}} \binom{d}{i_{d-r+1}}.
\]
Consequently,
\begin{equation}
|P(B_t)| \geq m_{d,r}(t) = \sum_{i_0=0}^t \sum_{i_1=0}^{i_0} \sum_{i_2=0}^{i_1} \cdots \sum_{i_{d-r+1}=0}^{i_{d-r}} \binom{d}{i_{d-r+1}}. \tag*{\qed}
\end{equation}
\end{corollary}

Corollary \ref{co:min} is not quite what we have so far described as the first extremal result. It determines an upper bound for $\ex_{d,r}(t)$, namely $m_{d,r}(t)$, but we do not yet know that $\ex_{d,r}(t)=m_{d,r}(t)$. In order to prove that, we shall show that the size of a set that we know protects the origin, namely the set defined in \eqref{eq:extrset}, is equal to $m_{d,r}(t)$. Our next aim is to verify that claim, but before we do, we define the extremal sets in full generality.

A subset $K$ of $B_t(x)$ is \emph{$(d,r)$-canonical} if there is a subset $I$ of $[d]$ of size $r-1$ and $\epsilon_i\in\{-1,1\}$ for each $i\in I$ such that
\[
K =\{y\in B_t(x) : y_i-x_i\in\{0,\epsilon_i\} \text{ for all } i\in I\}.
\]
The elements of the set $K$ are called \emph{canonical sites}. The ball $B_t(x)$ will often be implicit, but where needed we call the parameter $t$ the \emph{radius} of $K$. The set $I$ is the \emph{orientation} of $K$ and its complement $[d]\setminus I$ is the \emph{alignment} of $K$. A coordinate $i\in I$ is an \emph{orientation coordinate} and a coordinate $j\in[d]\setminus I$ is an \emph{alignment coordinate}.

Canonical sets, as just defined and as in \eqref{eq:extrset}, are the natural candidates for extremal sets. Unfortunately they are not the only examples, although the other examples, of which there are only a constant number, only differ in the positions of the sites of degree $1$. More specifically, any extremal set can be obtained from a $(d,r)$-canonical set $K$ by applying the following algorithm. For each site $x\in K$ that has degree inside $K$ equal to $1$, and its unique neighbour $y\in K$, either keep $x$ or replace it by any other neighbour of $y$ not already in $K$. The following definitions formalize this.

Given $x$, and $I$ and $\epsilon_i$ as above, let
\[
E_j^+ = \{x + te_j\} \cup \{x + (t-1)e_j -\epsilon_i e_i : i\in I\}
\]
for $j\in[d]\setminus I$, and similarly let
\[
E_j^- = \{x - te_j\} \cup \{x - (t-1)e_j -\epsilon_i e_i : i\in I\}.
\]
for $j\in[d]\setminus I$. Let $E$ be any set consisting of exactly one site from each of the $E_j^+$ and each of the $E_j^-$, so $|E|=2(d-r+1)$. A subset $K'$ of $B_t(x)$ is \emph{$(d,r)$-semi-canonical} if there is a $(d,r)$-canonical set $K$ and a choice of $E$ (with the alignment and orientation given by $K$) such that
\[
K' = \big( K \setminus \{x+te_j,x-te_j:j\in[d]\setminus I\} \big) \cup E.
\]
We call the sites in $E$ the \emph{extreme sites} of the $(d,r)$-semi-canonical set $K'$.

Observe that there are
\[
g_{d,r} = \binom{d}{d-r+1} 2^{r-1} d^{2(d-r+1)}
\]
$(d,r)$-semi-canonical sets: $\binom{d}{d-r+1}$ choices of orientation $I$, $2^{r-1}$ choices of the $\epsilon_i$, and $d$ choices for each of the $2(d-r+1)$ extreme sites.

\begin{lemma}\label{le:canonsize}
Let $t\in\N$ and let $2\leq r\leq d$. Suppose $K\subset B_t$ is $(d,r)$-canonical. Then $|K| = m_{d,r}(t)$.
\end{lemma}

Note that a $(d,r)$-semi-canonical set has the same size as a $(d,r)$-canonical set (with the same radius), so Lemma \ref{le:canonsize} also applies to these sets.

\begin{proof}
The induction is on $d$. Let $K_k$ be the intersection of $K$ with $S_k$. We shall prove that
\[
|K_k| = \sum_{i_1=0}^k \sum_{i_2=0}^{i_1} \cdots \sum_{i_{d-r+1}=0}^{i_{d-r}} \binom{d}{i_{d-r+1}}
\]
for $k=0,\dots,t$, which will prove the lemma.

Let the alignment set be $[d-r+1]$ and the orientation set $[d]\setminus[d-r+1]$. When $d=r$ there is exactly one alignment coordinate, so
\[
|K_k| = \binom{r-1}{k} + 2\sum_{i=0}^{k-1} \binom{r-1}{i} = \sum_{i=0}^k \binom{r}{i},
\]
the second equality following from Lemma \ref{le:minibinom}.

Suppose the lemma holds for $d-1$. The set of sites in $K_k$ with first coordinate zero is a $(d-1,r)$-canonical set, so the number of such sites is
\[
\sum_{i_2=0}^k \cdots \sum_{i_{d-r+1}=0}^{i_{d-r}} \binom{d-1}{i_{d-r+1}}
\]
by the induction hypothesis. Now fix $i$ such that $0\leq i\leq k-1$; negative $i$ are treated similarly by symmetry. The set of sites in $K_k$ with first coordinate $k-i$ is a $(d-1,r)$-canonical set of radius $i$ in the affine subspace $\{x:x_1=k-i\}$, so by induction the number of such sites is
\[
\sum_{i_2=0}^i \cdots \sum_{i_{d-r+1}=0}^{i_{d-r}} \binom{d-1}{i_{d-r+1}}.
\]
Summing over $i$ and doubling to take account of the choice of sign of $x_1$, the number of sites in $K_k$ with first coordinate non-zero is
\[
2\sum_{i_1=0}^{k-1} \sum_{i_2=0}^{i_1} \cdots \sum_{i_{d-r+1}=0}^{i_{d-r}} \binom{d-1}{i_{d-r+1}}.
\]
The result now follows from Lemma \ref{le:minibinom}.
\end{proof}

Given that the origin is protected, we say that the set $P(B_t)$ is \emph{minimal} if it has size $m_{d,r}(t)$. For any minimal set $P(B_t)$, we define $l_{d,r}(t)$ to be the number of sites in $P(S_t)$, which by Lemmas \ref{le:key} and \ref{le:canonsize} is equal to $m_{d,r+1}(t)$ and to the size of the intersection of a $(d,r)$-canonical set with $S_t$. Thus,
\[
l_{d,r}(t) = \sum_{i_1=0}^t \sum_{i_2=0}^{i_1} \cdots \sum_{i_{d-r+1}=0}^{i_{d-r}} \binom{d}{i_{d-r+1}}.
\]
Usually $d$ and $r$ will be clear from the context, so we shall write $l(t)$ and $m(t)$ for $l_{d,r}(t)$ and $m_{d,r}(t)$ respectively. Given that the origin is protected, the set $P(S_t)$ is \emph{minimal} if it has size $l_{d,r}(t)$.

The following result, which is a combination of Corollary \ref{co:min} and Lemma \ref{le:canonsize}, is the full version of the first extremal theorem.

\begin{theorem}\label{th:min}
Let $t\in\N$ and let $2\leq r\leq d$. Suppose that the origin is protected under $r$-neighbour bootstrap percolation. Then $|P(S_k)| \geq l_{d,r}(k)$ for $k=0,\dots,t$, and
\[
|P(B_t)| \geq |K| = m_{d,r}(t),
\]
where $K$ is any $(d,r)$-canonical set.
\end{theorem}

The next theorem states that all minimal sets protecting the origin are semi-canonical. It is the second extremal theorem.

\begin{theorem}\label{th:extremal}
Let $t\geq 2$ and let $2\leq r\leq d$. Suppose that the origin is protected under $r$-neighbour bootstrap percolation and that $P(B_t)$ is minimal. Then $P(B_t)$ is $(d,r)$-semi-canonical.
\end{theorem}

\begin{proof}
Throughout the proof we write $P_k$ for $P(S_k)$ for each $k\in[t]$. Since the origin is protected, it must have at least $2d-r+1$ protected neighbours. Therefore, by the pigeonhole principle, it must have at least $d-r+1$ pairs of opposing protected neighbours. Suppose it has at least $d-r+2$ pairs of opposing protected neighbours. Let $R_2$ be the set of sites in $S_2$ which have degree at least $2$ (and hence exactly $2$) in $P_1$. $R_2$ is precisely the set of all $x+y$ such that $x$ and $y$ belong to $P_1$ and $x+y\neq 0$. Thus,
\begin{equation}\label{eq:r2}
|R_2| \leq \binom{2d-r+1}{2} - (d-r+2).
\end{equation}
It is easy to verify (for example, by using Lemma \ref{le:minibinom} and induction on $d-r$) that
\[
l_{d,r}(2) = \sum_{i_1=0}^2 \sum_{i_2=0}^{i_1} \dots \sum_{i_{d-r+1}=0}^{i_{d-r}} \binom{d}{i_{d-r+1}} = \binom{2d-r+1}{2} + (d-r+1),
\]
which means that the right-hand side of \eqref{eq:r2} is exactly $2(d-r+1)+1$ less than $l_{d,r}(2)$. Therefore,
\begin{equation}\label{eq:r2-2}
\sum_{x\in P_2} d_{P_1}(x) \leq \sum_{x\in R_2} d_{P_1}(x) + 2(d-r+1)+1 = 2|R_2| + 2(d-r+1)+1.
\end{equation}
Every $x\in P_1$ has at least $2d-r+1$ protected neighbours, of which one is the origin and the rest are in $S_2$. The total out degree of $P_1$ is thus
\begin{equation}\label{eq:r2-3}
\sum_{x\in P_1} d_{P_2}(x) = (2d-r)|P_1| = (2d-r)(2d-r+1).
\end{equation}
The combination of \eqref{eq:r2}, \eqref{eq:r2-2} and \eqref{eq:r2-3} is a contradiction. We conclude that the origin has exactly $d-r+1$ pairs of opposing protected neighbours. Without loss of generality, let
\begin{equation}\label{eq:P1}
P_1 = \{e_1,\dots,e_d,-e_1,\dots,-e_{d-r+1}\}.
\end{equation}

This is the base case of our induction. There are two parts to the remainder of the proof. In Part (A), we show that if $P_{k-1}$ is $(d,r)$-canonical then $P_k$ must also be $(d,r)$-semi-canonical. Later, in Part (B), we show that if a set of sites $P_{k-1}'$ in $B_{k-1}$ is semi-canonical but not canonical then we cannot extend $P_{k-1}'$ to the $k$th sphere.

\emph{Part (A).} In this part of the proof we assume that $P_{k-1}$ is $(d,r)$-canonical. For the induction step, the idea is as follows. On the one hand, every site in $P_{k-1}$ must have at least a certain fixed number of edges into $P_k$ to ensure that it is protected, while on the other hand, we are only allowed a certain fixed number of sites in $P_k$ to achieve this, so we need to choose relatively few sites in $S_k$ with relatively large degree into $P_{k-1}$. We shall show that the canonical vertices in $S_k$ (or more precisely, those with degree at least $2$ into $P_{k-1}$) are the only vertices in $S_k$ with enough edges into $P_{k-1}$ to achieve this aim.

Let $R_k$ be the set of sites in $S_k$ with degree at least $2$ in $P_{k-1}$. We know exactly what $P_{k-1}$ is: it is the intersection of $S_{k-1}$ with the $(d,r)$-canonical set having alignment coordinates $1,\dots,d-r+1$ and $\epsilon_{d-r+2}=\dots=\epsilon_d=1$. Therefore, we can also easily see exactly what $R_k$ is: it is the intersection of $S_k$ with the corresponding $(d,r)$-canonical set that has radius $k$, excluding its extreme sites. It follows that $|R_k|$ is exactly $2(d-r+1)$ less than $l_{d,r}(k)$, by Lemma \ref{le:canonsize}. The protected set $P_k$ is minimal and therefore has size $l_{d,r}(k)$ by Theorem \ref{th:min}, so we have
\begin{equation}\label{eq:PkRk}
|P_k|=|R_k|+2(d-r+1).
\end{equation}
Since $R_k$ contains every site in $S_k$ that has degree at least $2$ into $P_{k-1}$, and $R_k$ has size exactly $2(d-r+1)$ less than the extremal number for the sphere $S_k$, we must have
\begin{equation}\label{eq:configscompare}
\sum_{x\in P_k} d_{P_{k-1}}(x) \leq \sum_{x\in R_k} d_{P_{k-1}}(x) + 2(d-r+1),
\end{equation}
and if there is equality in \eqref{eq:configscompare} then it must be the case that $R_k\subset P_k$. Also, we have just observed that $R_k$ consists of all non-extreme canonical sites in $S_k$. Together with the minimality of $P_k$ and Theorem \ref{th:min}, this implies that
\begin{equation}\label{eq:configsextreme}
\sum_{x\in R_k} d_{P_{k-1}}(x) = \sum_{x\in P_{k-1}} d_{P_k}(x) - 2(d-r+1).
\end{equation}
Thus we have equality in \eqref{eq:configscompare}, and as noted above, this means that $R_k\subset P_k$. The only sites in $P_{k-1}$ that do not have enough protected neighbours in $P_k$ so far are the extreme sites, and they all need exactly one more protected neighbour. It follows that $P_k$ must be $(d,r)$-semi-canonical. This completes Part (A) of the proof.

\emph{Part (B).} In this part of the proof we let $P_{k-1}$ and $R_k$ be as before (that is, a canonical set in $B_{k-1}$ and the set of sites in $S_k$ that have degree at least $2$ into $P_{k-1}$, respectively), and we let $P_{k-1}'$ be any semi-canonical (but not canonical) set obtained from $P_{k-1}$ by changing one or more extreme sites. Thus, $P_{k-1}'$ has $m\geq 1$ extreme sites $x_1,\dots,x_m$ not of the form $\pm (k-1)e_i$ for any $i$, and the remaining extreme sites $y_1,\dots,y_{2(d-r+1)-m}$ are of that form. Our aim will be to show that if the set of protected sites in $B_{k-1}$ is precisely $P_{k-1}'$, and $P(B_k)$ is minimal, then there is no way to extend $P_{k-1}'$ into the $k$th sphere, which would be a contradiction. Our method for proving this is to show that each $x_i$ requires at least one more protected neighbour in $S_k$ than the site it replaced in $P_{k-1}$, and by comparison with the inequalities derived above for $P_{k-1}$, we deduce that there would need to be more than the minimum number of protected sites in $S_k$ to extend $P_{k-1}'$ into the $k$th sphere, a contradiction.

Suppose on the contrary that there is a minimal set $P_k'$ in $S_k$ extending $P_{k-1}'$. Let $R_k'$ be the set of sites in $S_k$ with degree at least $2$ in $P_{k-1}'$. Observe that $R_k\setminus R_k'$ consists precisely of the sites in $R_k$ that neighbour one of the extreme sites in $P_{k-1}\setminus P_{k-1}'$. There are $m$ extreme sites in $P_{k-1}\setminus P_{k-1}'$ and each has exactly $2d-r-1$ neighbours in $R_k$: two for each of the $d-r+1$ alignment coordinates, except that of the extreme site itself, and one for each of the $r-1$ orientation coordinates. Therefore we have $|R_k\setminus R_k'| = m(2d-r-1)$. On the other hand, $R_k'\setminus R_k$ consists precisely of the sites in $R_k'$ that neighbour one of the extreme sites in $P_{k-1}'\setminus P_{k-1}$. Again, there are $m$ extreme sites in $P_{k-1}'\setminus P_{k-1}$, but here each only has $2d-r-2$ neighbours in $R_k'$: if the site is $(k-2)e_i-e_j$, say, then it has two neighbours in $R_k'$ for each alignment coordinate except $i$, and one for each orientation coordinate except $j$. It follows that $|R_k'\setminus R_k|=m(2d-r-2)$, and therefore
\[
|R_k'| = |R_k|-m.
\]
From \eqref{eq:PkRk} we know that $|P_k| = |R_k| + 2(d-r+1)$, and so
\[
|P_k'| = |P_k| = |R_k| + 2(d-r+1) = |R_k'| + 2(d-r+1) + m.
\]
Therefore, as in \eqref{eq:configscompare},
\begin{equation}\label{eq:semicompare}
\sum_{x\in P_k'} d_{P_{k-1}'}(x) \leq \sum_{x\in R_k'} d_{P_{k-1}'}(x) + 2(d-r+1) + m.
\end{equation}

Next, observe that every site in $R_k\setminus R_k'$ has degree exactly $2$ into $P_{k-1}$ and that every site in $R_k'\setminus R_k$ has degree exactly $2$ into $P_{k-1}'$. From this it follows that
\begin{equation}\label{eq:semitransfer1}
\sum_{x\in R_k'} d_{P_{k-1}'}(x) = \sum_{x\in R_k} d_{P_{k-1}}(x) - 2m.
\end{equation}

The final ingredient we need to obtain a contradiction is the observation that
\begin{equation}\label{eq:semitransfer2}
\sum_{x\in P_{k-1}} d_{P_k}(x) \leq \sum_{x\in P_{k-1}'} d_{P_k'}(x).
\end{equation}
This inequality holds because every site (necessarily extreme) in $P_{k-1}\setminus P_{k-1}'$ has degree exactly $2d-r$ into $P_k$, and every site in $P_{k-1}'\setminus P_{k-1}$ (also necessarily extreme) has degree exactly $1$ into $P_{k-2}$, and hence degree at least $2d-r$ into $P_k'$.

The combination of \eqref{eq:semicompare}, \eqref{eq:semitransfer1}, \eqref{eq:configscompare} and \eqref{eq:semitransfer2} gives
\[
\sum_{x\in P_k'} d_{P_{k-1}'}(x) \leq \sum_{x\in P_{k-1}'} d_{P_k'}(x) - m,
\]
a contradiction if $m\geq 1$.
\end{proof}

\section{Near-minimal configurations}\label{se:stab}

We turn to the third of the three extremal theorems, which gives a rough description of near-minimal protecting sets. The theorem is as follows.

\begin{theorem}\label{th:stability}
Let $2\leq r\leq d$. There exist $c_1$ and $c_2$ depending only on $d$ such that the following holds. Suppose there exists $k_1\geq c_1$ such that $P(S_k)$ is minimal for all $k$ in the range $k_1\leq k\leq k_1+c_2$. Let $t\geq k_1+c_2$ and suppose that the origin is protected. Then $P(B_{k_1-c_1})$ is $(d,r)$-canonical.
\end{theorem}

The theorem implies the following statement. If the origin is protected and the protected sites in at least a fixed constant number of spheres are minimal, then the protected sites in all but a final (different) fixed constant number of layers are a canonical set. The corollary of this that we need in the proof of Theorem \ref{th:main} is that if the origin is protected and there are $m(t)+a$ protected sites in $B_t$, then the number of possible configurations of the protected sites is $t^{O(a)}$. The trivial bound would be $t^{O(ta)}$.

Throughout the proof we make frequent use of variations on the following idea, which we refer to as the \emph{hyperplane restriction principle}. We start by assuming that $S_k$ is minimal. Suppose we know that both $-e_i$ and $e_i$ are protected. Let $C_1^+$ be the compatibility function given by $C_1^+(1)=1$ and $C_1^+(i)=\ast$ for $i\neq 1$, and let $C_1^-$ be the compatibility function given by $C_1^-(1)=-1$ and $C_1^-(i)=\ast$ for $i\neq 1$. By Lemma \ref{le:key},
\[
|P_{k-1}^{C_1^+}(e_1)|  \geq \sum_{i_1=0}^{k-1} \sum_{i_2=0}^{i_1} \cdots \sum_{i_{d-r+1}=0}^{i_{d-r}} \binom{d-1}{i_{d-r+1}},
\]
and the same inequality holds for $|P_{k-1}^{C_1^-}(-e_1)|$. Since $P(S_k)$ is minimal, there are exactly
\[
\sum_{i_1=0}^k \sum_{i_2=0}^{i_1} \cdots \sum_{i_{d-r+1}=0}^{i_{d-r}} \binom{d}{i_{d-r+1}}
\]
protected sites at distance $k$ from the origin. Let $C_1^0$ be the $1$-restriction of $C_1^+$ (equivalently, of $C_1^-$). Noting that the sets $P_{k-1}^{C_1^+}(e_1)$, $P_{k-1}^{C_1^-}(-e_1)$ and $P_k^{C_1^0}(0)$ partition the set of protected sites in $S_k$, we have from the above inequalities and Lemma \ref{le:minibinom} that
\begin{equation}\label{eq:stabobs}
|P_k^{C_1^0}(0)| \leq \sum_{i_2=0}^k \cdots \sum_{i_{d-r+1}=0}^{i_{d-r}} \binom{d-1}{i_{d-r+1}}.
\end{equation}
In other words, by assuming that $-e_1$ and $e_1$ are protected, we obtain an upper bound on the number of protected sites $x$ with $x_1=0$ that is equal to the minimum it could be. Why is this a useful observation? It is the key to applying induction: it gives minimality of $S_k$ in a codimension $1$ subspace of $\Z^d$, so we may deduce that the set of protected sites inside that subspace must be $(d-1,r)$-canonical. The principle is illustrated in Figure \ref{fi:hrp}.

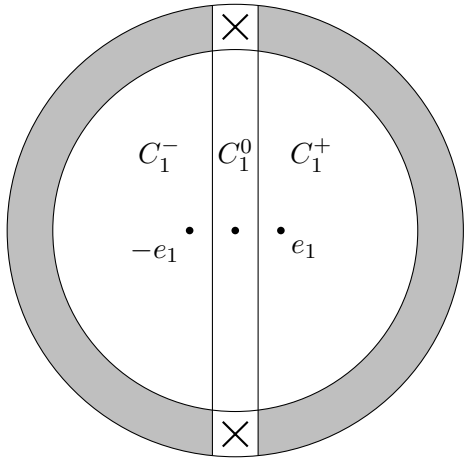
\begin{figure}[ht]
  \centering
  \begin{tikzpicture}[>=latex]
    \path [name path=C1,fill=gray!50] (0,0) circle (3);
    \path [name path=C2,fill=white] (0,0) circle (2.4);
    \path [name path=L1] (-0.3,0) -- (-0.3,3);
    \path [name path=L2] (-0.3,0) -- (-0.3,-3);
    \path [name path=L3] (0.3,0) -- (0.3,3);
    \path [name path=L4] (0.3,0) -- (0.3,-3);
    \path [name intersections={of=C1 and L1,by=P1}];
    \path [name intersections={of=C1 and L2,by=P2}];
    \path [name intersections={of=C1 and L3,by=P3}];
    \path [name intersections={of=C1 and L4,by=P4}];
    \fill [white] (-0.3,-3) rectangle (0.3,3);
    \draw (0,0) circle (3) (0,0) circle (2.4);
    \draw (P1) -- (P2) (P3) -- (P4);
    \node [circle,fill,inner sep=0,minimum size=0.1cm] at (0,0) {};
    \draw (-0.6,0) node [circle,fill,inner sep=0,minimum size=0.1cm] {} node [below left] {$-e_1$};
    \draw (0.6,0) node [circle,fill,inner sep=0,minimum size=0.1cm] {} node [below right] {$e_1$};
    \node at (-1,1) {$C_1^-$};
    \node at (0,1) {$C_1^0$};
    \node at (1,1) {$C_1^+$};
    \node [cross out,draw,thick,inner sep=0,minimum size=0.3cm] at (0,2.7) {};
    \node [cross out,draw,thick,inner sep=0,minimum size=0.3cm] at (0,-2.7) {};
  \end{tikzpicture}
  \caption{This figure depicts the hyperplane restriction principle. Assuming that $-e_1$ and $e_1$ are protected, Lemma \ref{le:key} implies that there are many protected sites in the two grey arcs, which represent the intersection of $S_k$ with the sites that are $C_1^-$-compatible with $-e_1$ (on the left) and $C_1^+$-compatible with $e_1$ (on the right). Assuming that $P(S_k)$ is minimal and that the origin is also protected, Lemma \ref{le:minibinom} then implies that the intersection of $S_k$ with the hyperplane $\{x:x_1=0\}$, shown here as the two crosses, is minimal.}
  \label{fi:hrp}
\end{figure}

The hyperplane restriction principle we have just described can be applied much more widely than as here to a hyperplane passing through the origin. We shall see in the proof of Theorem \ref{th:stability} that the calculation works when an arbitrary site plays the role played here by the origin.

We now sketch the proof of Theorem \ref{th:stability}, which is by induction on $d-r$. First, observe that the origin has at least $d-r+1$ pairs of opposing protected neighbours, which we may assume are $\pm e_1,\dots,\pm e_{d-r+1}$. Define for each $i$ the hyperplanes
\begin{equation}\label{eq:UiVi}
U_i = \{x\in\Z^d:x_1=i\} \qquad \text{and} \qquad V_i = \{x\in\Z^d:x_2=i\}.
\end{equation}
The hyperplane restriction principle shows that $U_i\cap S_k$ is minimal (under the $r$-neighbour model in $d-1$ dimensions, of course) for a large range of $k$, and the induction hypothesis then provides us with a $(d-1,r)$-canonical set of protected sites in $U_0$, of a suitable radius. The only corollary of this that we use is that $ie_2$ is protected for a large range of $i$. The same idea provides a $(d-1,r)$-canonical protected set in $V_0$, but now, because we know the sites $ie_2$ are protected, we can continue to apply the hyperplan restriction principle to obtain $(d-1,r)$-canonical protected sets $P_i$ of appropriate radii in each $V_i$.

Our remaining task is two-fold. First, we must show that the $P_i$ have the same set of alignment coordinates, namely $[d-r+1]$, and second we must show that they have the same orientation. For the first, we again use the hyperplane restriction principle. Since $\pm e_j$ are both protected for each $j\leq d-r+1$, we obtain an upper bound on the size of the set of protected sites in the hyperplane $\{x:x_j=0\}$ for each $i$ in the same range. If for one of the $P_i$ there is a coordinate $j\leq d-r+1$ that is not an alignment coordinate then we find too many protected sites in the hyperplane $\{x:x_j=0\}$. For the second claim, observe that we could have interchanged the roles of $e_1$ and $e_2$ in everything we have done so far in the proof, and obtain $(d-1,r)$-canonical protected sets in the hyperplanes $U_i$. It is then a simple counting argument to show that the union of the $U_i$ must equal the union of the $V_i$, and that this implies at the orientations match.

\begin{proof}[Proof of Theorem \ref{th:stability}]
The proof is by induction on $d-r$. The base case $d-r=0$ is Theorem $11$ in \cite{BHSU}, for which the proof gives constants $c_1=d$ and $c_2=3d+1$. Here we assume that $d-r\geq 1$ and we take as our hypothesis that the result holds for smaller values of $d-r$ with the same constants $c_1$ and $c_2$.

In what follows, $k$ will always be assumed to be in the range $k_1\leq k \leq k_1+c_2$. For each $i$, let $U_i$ be the hyperplane as in \eqref{eq:UiVi}. The origin can have at most $r-1$ neighbours that are not protected, so it must have at least $2d-r+1\geq d+2$ neighbours that are protected. Therefore we may assume that $\pm e_1,\dots,\pm e_{d-r+1}$ are all protected.

The hyperplane restriction principle, with precisely the example given in the preamble to the proof, implies that the set of protected sites in the $(d-2)$-dimensional sphere $S_k\cap U_0$ embedded in the $(d-1)$-dimensional space $U_0$ is minimal. This holds for each $k$ for which $S_k$ is minimal, so it certainly is true for all $k$ satisfying $k_1\leq k\leq k_1+c_2$. By induction, it follows that $P(B_{k_1-c_1}\cap U_0)$ is a $(d-1,r)$-canonical set. In particular this means that $ie_2$ is protected for $i=-k_2,\dots,k_2$, where $k_2=k_1-c_1$.

Next we show inductively that the intersections of $P(S_k)$ with the hyperplanes $V_i$ (defined in \eqref{eq:UiVi}) are minimal for $i=-k_2,\dots,k_2$. First, the set $P(S_k\cap V_0)$ is minimal by the same argument we used to prove $P(S_k\cap U_0)$ is minimal. So by symmetry, we just have to show that $P(S_k\cap V_i)$ is minimal for $i=1,\dots,k_2$.

We use the hyperplane restriction principle again, this time recursively. Let $C_2^+$ be the compatibility function given by $C_2^+(2)=1$ and $C_2^+(j)=\ast$ for $j\neq 2$, let $C_2^-$ be the compatibility function given by $C_2^-(2)=-1$ and $C_2^-(j)=\ast$ for $j\neq 2$, and let $C_2^0$ be the $2$-restriction of $C_2^+$. For the induction to go through, we make a stronger claim: that $P(S_k\cap V_i)$ is minimal for smaller values of $i$ and also that
\begin{equation}\label{eq:configs1}
|P_{k-i}^{C_2^+}(ie_2)| = \sum_{i_1=0}^{k-i} \sum_{i_2=0}^{i_1} \cdots \sum_{i_{d-r+1}=0}^{i_{d-r}} \binom{d-1}{i_{d-r+1}}
\end{equation}
holds for smaller values of $i$. This is the case when $i=0$ because
\[
|P_k^{C_2^-}(0)|  \geq \sum_{i_1=0}^k \sum_{i_2=0}^{i_1} \cdots \sum_{i_{d-r+1}=0}^{i_{d-r}} \binom{d-1}{i_{d-r+1}}
\]
by Lemma \ref{le:key}, so by minimality of $P(S_k)$ we must have equality here and in \eqref{eq:configs1} in the case $i=0$. Suppose the claim holds for $i-1$. Then
\begin{align}
|P_{k-i}^{C_2^0}(ie_2)| &= |P_{k-i}^{C_2^+}(ie_2)| - |P_{k-i-1}^{C_2^+}((i+1)e_2)| \notag \\
&\leq \sum_{i_1=0}^{k-i} \sum_{i_2=0}^{i_1} \cdots \sum_{i_{d-r+1}=0}^{i_{d-r}} \binom{d-1}{i_{d-r+1}} - \sum_{i_1=0}^{k-i-1} \sum_{i_2=0}^{i_1} \cdots \sum_{i_{d-r+1}=0}^{i_{d-r}} \binom{d-1}{i_{d-r+1}} \label{eq:configs2} \\
&= \sum_{i_2=0}^{k-i} \sum_{i_3=0}^{1_2} \cdots \sum_{i_{d-r+1}=0}^{i_{d-r}} \binom{d-1}{i_{d-r+1}}. \notag
\end{align}
So $P(S_k\cap V_i)$ is minimal. Furthermore, we must have equality in \eqref{eq:configs2}, so \eqref{eq:configs1} holds for $i+1$. This completes the proof of the claim that $P(S_k\cap V_i)$ is minimal for $i=-k_2,\dots,k_2$.

By induction, $P(B_{k_2}\cap V_i)$ is $(d-1,r)$-canonical for $i=-k_2,\dots,k_2$. To save space, we shall write $P_i$ for $P(B_{k_2}\cap V_i)$. To complete the proof of the theorem we have to show two things: Claim (A), that each $P_i$ has the same alignment, and Claim (B), that each $P_i$ has the same orientation. These claims are sufficient to prove the theorem because together they imply that $P_{-k_2}\cup\dots\cup P_{k_2}$ is a $(d,r)$-canonical set of protected sites, as desired.

\begin{figure}[ht]
  \centering
  \begin{tikzpicture}[scale=0.5,>=latex]
    \fill[gray!50] (-5,5) rectangle (6,6);
    \draw (-5,5) rectangle (-4,6) (-4,4) rectangle (-3,7) (-3,3) rectangle (-2,8) (-2,2) rectangle (-1,9) (-1,1) rectangle (0,10) (0,0) rectangle (1,11) (1,1) rectangle (2,10) (2,2) rectangle (3,9) (3,3) rectangle (4,8) (4,4) rectangle (5,7) (5,5) rectangle (6,6);
    \draw (-5,5) rectangle (6,6);
    \node [cross out,draw,thick,inner sep=0,minimum size=0.1cm] at (0.5,5.5) {};
  \end{tikzpicture}
  \caption{The filled left-right strip in the figure depicts the set of protected sites $\{ie_2:|i|\leq k_2\}$. The vertical strips depict the intersections of the hyperplanes $V_i$ with the ball $B_{k_2}$. Induction shows that the set of protected sites $P_i$ inside the $i$th strip is $(d-1,r)$-canonical. Since $\pm e_1,\dots, \pm e_{d-r+1}$ are all protected, multiple applications of the hyperplane restriction principle imply that none of the hyperplanes $\{x\in\Z^d:x_j=0\}$ can contain too many protected sites, for $j\leq d-r+1$. On the other hand, if one of the $P_i$ has alignment coordinates other than $[d-r+1]$ then that forces too many protected sites in one of the hyperplanes.}
  \label{fi:stabalign}
\end{figure}
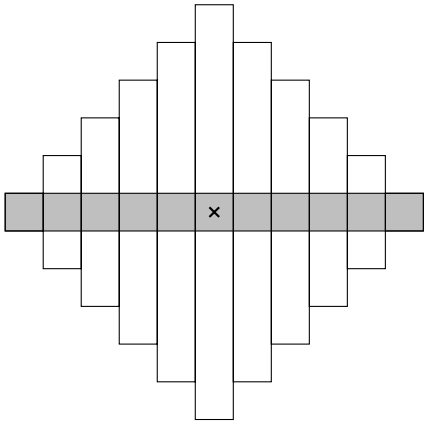

\emph{Claim (A).} We start with the claim that each $P_i$ has the same alignment. We shall show that if the set of alignment coordinates of one of the $P_i$ is not $[d-r+1]$ then there are too many protected sites in $S_k$; see Figure \ref{fi:stabalign}. Suppose there is a choice of $j\in[d-r+1]$ and $i\in\{-k+1,\dots,k-1\}$ such that $j$ is not a direction of alignment for $P_i$. We may assume that $d-r+2$ is a direction of alignment for this set instead. Again we use a variation on the hyperplane restriction principle. Since both $e_j$ and $-e_j$ are protected, Lemma \ref{le:key} tells us that there are at least
\[
2\sum_{i_1=0}^{k-1} \sum_{i_2=0}^{i_1} \cdots \sum_{i_{d-r+1}=0}^{i_{d-r}} \binom{d-1}{i_{d-r+1}}
\]
protected sites $x$ in $S_k$ with $x_j\neq 0$. By minimality there are exactly
\[
\sum_{i_1=0}^k \sum_{i_2=0}^{i_1} \cdots \sum_{i_{d-r+1}=0}^{i_{d-r}} \binom{d}{i_{d-r+1}}
\]
protected sites in $S_k$ in total, so by Lemma \ref{le:minibinom} there are at most
\begin{equation}\label{eq:stabxj=0}
\sum_{i_2=0}^k \sum_{i_3=0}^{i_2} \cdots \sum_{i_{d-r+1}=0}^{i_{d-r}} \binom{d-1}{i_{d-r+1}}
\end{equation}
protected sites $x$ in $S_k$ with $x_j=0$. (In fact we have equality, but we do not need that.) For $-k\leq l\leq k$, let
\[
Q_l = \{x\in P_l : x_j=0\}.
\]
If $j$ is an alignment coordinate for $P_l$ then $Q_l$ is $(d-2,r)$-canonical, while if $j$ is an orientation coordinate for $P_l$ then $Q_l$ is $(d-2,r-1)$-canonical. In either case,
\[
|Q_l\cap S_k| \geq \sum_{i_3=0}^{k-l} \sum_{i_4=0}^{i_3} \cdots \sum_{i_{d-r+1}=0}^{i_{d-r}} \binom{d-2}{i_{d-r+1}}.
\]
Also, $Q_i\cap S_k$ contains at least one more site than this, because by assumption $j$ is an orientation coordinate for $P_i$, so $Q_i$ is $(d-2,r-1)$-canonical and therefore
\[
|Q_i\cap S_k| \geq \sum_{i_2=0}^{k-i} \sum_{i_3=0}^{i_2} \cdots \sum_{i_{d-r+1}=0}^{i_{d-r}} \binom{d-2}{i_{d-r+1}} \geq \sum_{i_3=0}^{k-i} \sum_{i_4=0}^{i_3} \cdots \sum_{i_{d-r+1}=0}^{i_{d-r}} \binom{d-2}{i_{d-r+1}} + 1.
\]
Hence, by Lemma \ref{le:minibinom},
\[
\sum_{l=-k}^k |Q_l\cap S_k| \geq \sum_{i_2=0}^k \sum_{i_3=0}^{i_2} \cdots \sum_{i_{d-r+1}=0}^{i_{d-r}} \binom{d-1}{i_{d-r+1}} +1,
\]
contradicting \eqref{eq:stabxj=0}. This completes the proof of Claim (A).

\emph{Claim (B).} Next we prove the second claim, that each of the $(d-1,r)$-canonical sets has the same orientation. Note that we could have reversed the roles of $e_1$ and $e_2$ in everything we have done so far and obtained $(d-1,r)$-canonical sets $R_{-k},\dots,R_k$ in $P(B_k)\cap U_{-k},\dots,P(B_k)\cap U_k$ respectively. These canonical sets have the same alignment as the canonical sets for the $U_k$, by the same reasoning as above. The set $(P_{-k}\cup\dots\cup P_k) \cap B_k$ is equal to $P(B_k)$, the set of protected sites in $B_k$, and so is the set $(R_{-k}\cup\dots\cup R_k) \cap B_k$. Thus
\[
(P_{-k}\cup\dots\cup P_k) \cap B_k = (R_{-k}\cup\dots\cup R_k) \cap B_k.
\]
It now follows immediately that the orientations of the $P_i$ must match up: if some two are different then we are left with no choice of orientations for $R_0$. This completes the proof of the second claim, and of the theorem.
\end{proof}

\section{Proofs of main theorems}\label{se:main}

In the previous two sections we proved the three key extremal theorems. Here we use those theorems to derive good approximations to the first and second moments of the number of uninfected sites at time $t$. This will turn out to be key to proving Theorems \ref{th:main} and \ref{th:conc}.

Let us fix a sequence of probabilities $(p_n)_{n=1}^\infty$, let $E_{d,r}(t,n,x)$ be the event that the site $x\in\T_n^d$ is uninfected at time $t$, and let $F_{d,r}(t,n,x)$ be the corresponding indicator random variable. We are interested in the total number of uninfected sites at time $t$, defined to be $F_{d,r}(t,n)$; thus,
\[
F_{d,r}(t,n) = \sum_{x\in\T_n^d} F_{d,r}(t,n,x).
\]
Often we write $F(t,n)$ for $F_{d,r}(t,n)$. We would like to estimate the first two moments of $F(t,n)$. Together with the Stein-Chen method \cite{Stein,Chen}, this will allow us to prove that $F(t,n)$ is asymptotically Poisson distributed, which will enable us to complete the proof of Theorems \ref{th:main} and \ref{th:conc}. The version of Stein-Chen that we shall use is the following formulation due to Barbour and Eagleson \cite{BarEag}.

\begin{theorem}\label{th:steinchen}
Let $X_1,\dots,X_n$ be Bernoulli random variables with $\P (X_i = 1) = p_i.$ Let $Y_n = \sum_{i=1}^{n}{X_i}$, and let $\lambda_n = \mathbb{E}(Y_n) = \sum_{i=1}^n p_i.$ For each $i\in [n],$ let $N_i \subset [n]$ be such that $X_i$ is independent of $\{X_j:j\notin N_i\}$. For each $i,j \in [n],$ let $p_{ij} = \mathbb{P}(X_i X_j = 1)$. Let $Z_n \sim \Po(\lambda_n)$. Then 
\[
\sup_{A\subset\Z} \big| \P(Y_n\in A) - \P(Z_n\in A) \big| \leq \min\left\{1,\lambda_n^{-1}\right\} \Bigg( \sum_{i=1}^n \sum_{j\in N_i} p_i p_j + \sum_{i=1}^n \sum_{j\in N_i\setminus\{i\}} p_{ij} \Bigg).
\]
\end{theorem}

The expectation of $F(t,n)$ is $\E F(t,n) := \lambda(t,n) = n^d \rho_1$, where
\[
\rho_1 := \P_{p_n}(E_{d,r}(t,n,x)).
\]
We also need to bound the quantity
\[
\rho_2 := \max \big\{ \P_{p_n}\big(E_{d,r}(t,n,x) \cap E_{d,r}(t,n,y)\big) \, : \, \|x-y\| \leq 2t \big\},
\]
which we shall use to bound the $p_{ij}$ in Theorem \ref{th:steinchen}. The condition $\|x-y\| \leq 2t$ is equivalent to the statement that the events $E_{d,r}(t,n,x)$ and $E_{d,r}(t,n,y)$ are dependent.

The following lemma is a simple computation and is similar to Lemma 18 of \cite{BHSU}.

\begin{lemma}\label{le:tbound}
Let $t=o((\log n/\log\log n)^{d-r+1})$ and let
\begin{equation}\label{eq:qbound}
q_n = O(n^{-d/m_{d,r}(t)}\log n).
\end{equation}
Then for any constant $c>0$ we have $t^c q_n = o(1)$. \qed
\end{lemma}

\begin{lemma}\label{le:rho}
Let $t=o((\log n/\log\log n)^{d-r+1})$ and let $q_n$ satisfy \eqref{eq:qbound}. Then
\begin{equation}\label{eq:rho1}
\rho_1 = (1+o(1)) g_{d,r} q_n^{m_{d,r}(t)},
\end{equation}
where
\[
g_{d,r} = \binom{d}{d-r+1} 2^{r-1} d^{2(d-r+1)}.
\]
Furthermore,
\begin{equation}\label{eq:rho2}
\rho_2 = O(q_n\rho_1) = o(\rho_1).
\end{equation}
\end{lemma}

\begin{proof}
We only sketch the proof, since it is similar to the proofs of several lemmas in Section 4 of \cite{BHSU}. Let $g_{d,r}(t,k)$ be the number of arrangements of $m_{d,r}(t)+k$ uninfected sites in $B_t$ such that the origin is protected. Thus,
\begin{equation}\label{eq:rho1sum}
\rho_1 = \sum_{k=0}^{|B_t|-m_{d,r}(t)} g_{d,r}(t,k) p_n^{|B_t|-m_{d,r}(t)-k} q_n^{m_{d,r}(t)+k}.
\end{equation}
Theorem \ref{th:extremal} implies that
\[
g_{d,r}(t,0) = \binom{d}{d-r+1} 2^{r-1} d^{2(d-r+1)} = g_{d,r},
\]
while Theorem \ref{th:stability} allows us to bound $g_{d,r}(t,k)$ for general $k$ by
\begin{equation}\label{eq:gbound}
g_{d,r}(t,k) = t^{O(k)}.
\end{equation}
It is now easy to see that these last two equations combined with \eqref{eq:rho1sum} give the bound we want on $\rho_1$ in \eqref{eq:rho1}.

For \eqref{eq:rho2}, we follow Lemmas 18 and 19 of \cite{BHSU}. The key point is that a set of sites cannot be a semi-canonical set for two distinct sites, so if $x$ and $y$ are protected then $B_t(x)\cup B_t(y)$ contains at least $m_{d,r}(t)+1$ uninfected sites. This means that
\[
\rho_2 \leq \sum_{k=0}^{2|B_t|} h_{d,r}(t,k) q_n^{m(t)+1+k},
\]
where $h_{d,r}(t,k)$ is the number of configurations of $m(t)+1+k$ uninfected sites in $B_t(x)\cup B_t(y)$ such that both $x$ and $y$ are protected. Using the bound on $g_{d,r}(t,k)$ from \eqref{eq:gbound} we obtain a similar bound on $h_{d,r}(t,k)$, namely
\[
h_{d,r}(t,k) = O(t^{O(k)}).
\]
We are now able to estimate $\rho_2$ in much the same way that we estimated $\rho_1$.
\end{proof}

The following is an easy consequence of Lemma \ref{le:rho} and Theorem \ref{th:steinchen}.

\begin{theorem}\label{th:dist}
Let $2\leq r\leq d$, let $t=o((\log n/\log\log n)^{d-r+1})$, let $q_n$ satisfy \eqref{eq:qbound}, and let $X(t,n)\sim\Po(\lambda(t,n))$. Then
\[
\sup_{A\subset\Z} \big| \P(F(t,n)\in A) - \P(X(t,n)\in A) \big| = O(t^dq_n) = o(1). \tag*{\qedsymbol}
\]
\end{theorem}

\begin{proof}[Proof of Theorem \ref{th:main}]
The probability of percolating in time at most $t$ is increasing in $p$, by a standard coupling argument. Therefore, we may assume that the usual bound \eqref{eq:qbound} on $q_n$ holds. Theorem \ref{th:dist} tells us that $F(t,n)$ converges in distribution to $\Po(\lambda(t,n))$, so
\[
\P_{p_n}(T\leq t) = \P_{p_n}(F(t,n)=0) = (1+o(1))e^{-\lambda(t,n)}.
\]
We have
\[
\lambda(t,n) = n^d \rho_1 = n^d \Theta(1) q_n^{m_{d,r}(t)}
\]
by Lemma \ref{le:rho}. Therefore,
\[
\P_{p_n}(T\leq t) \to
\begin{cases}
1 & \text{if } q_n \leq (n^{-d}/\omega(n))^{1/m_{d,r}(t)}, \\
0 & \text{if } q_n \geq (n^{-d}\omega(n))^{1/m_{d,r}(t)},
\end{cases}
\]
for some function $\omega(n)\to\infty$, as required.
\end{proof}

\begin{proof}[Proof of Theorem \ref{th:conc}]
Part (i) of the theorem is an immediate corollary of Theorem \ref{th:main}. For part (ii), we are given that $q_n$ satisfies
\[
(n^{-d}/\omega(n))^{1/m_{d,r}(t)} \leq q_n \leq (n^{-d}\omega(n))^{1/m_{d,r}(t)}
\]
for all $\omega(n)\to\infty$. Observe that if
\[
\omega(n) = \exp\left(c_0\frac{\log n}{t}\right)
\]
for a sufficiently small constant $c_0$, then
\[
(n^{-d}/\omega(n))^{1/m_{d,r}(t)} \geq (n^{-d}\omega(n))^{1/m_{d,r}(t-1)}.
\]
Therefore, $\P_{p_n}(T\leq t-1) = o(1)$ by Theorem \ref{th:main}. Similarly we have $\P_{p_n}(T\geq t+2)=o(1)$. So $T\in\{t,t+1\}$ with high probability.

Now let $q_n^{m_{d,r}(t)}n^d \to c$ as $n\to\infty$. Then
\[
\P_{p_n}(T=t) \sim \P_{p_n}(T\leq t) \sim e^{-\lambda(t,n)} \sim \exp\big( -n^d g_{d,r} q_n^{m_{d,r}(t)} \big) \sim \exp(-g_{d,r}c).
\]
Since $T\in\{t,t+1\}$ with high probability, we must also have
\[
\P_{p_n}(T=t+1) \sim 1 - \exp(-g_{d,r}c). \qedhere
\]
\end{proof}

That completes the proofs of the main results. We now briefly turn our attention to the rather easier setting of subcritical models.

\section{Subcritical models}\label{se:sub}

In this final section we give a complete description of the time for percolation in the case of subcritical models. Here, subcritical means that there exist closed cofinite sets; in $d$ dimensions, this means that the threshold $r$ is strictly greater than $d$. Our description of the percolation time is valid for all $p$ for which percolation occurs with high probability; we determine how large $q=1-p$ needs to be for this to be the case, and as a corollary we determine the critical probabilities for percolation under subcritical models on the torus.

\begin{lemma}\label{le:sub}
Let $3\leq d+1\leq r\leq 2d$ and $k\geq 0$. Suppose the origin is protected under the $r$-neighbour model. Then
\begin{equation}\label{eq:subkey}
|P(B_k)| \geq \binom{2d-r+1}{k}.
\end{equation}
\end{lemma}

\begin{proof}
The proof is the usual inductive double counting argument. Write $P_k$ for $P(S_k)$ and consider the bipartite graph $H_k$ with vertex sets $P_{k-1}$ and $P_k$, and edges induced by $\Z^d$. Note that $|P_0|=1$, as claimed. If $x\in P_k$ then $x$ has at most $k$ non-zero coordinates, so the degree of $x$ in $H_k$ is at most $k$. If $y\in P_{k-1}$ then, since $y$ is protected under the $r$-neighbour model, it follows that $y$ can have at most $r-1$ infected neighbours at time $(t-k-1)$, and hence it has at least $2d-r+1$ uninfected neighbours. Of these, at most $k-1$ are in $P_{k-2}$, since $y$ has at most $k-1$ non-zero coordinates, so $y$ has at least $2d-r-k+2$ protected neighbours in $P_k$.

We have proved that
\[
k|P_k| \geq (2d-r-k+2)|P_{k-1}|,
\]
which implies \eqref{eq:subkey}.
\end{proof}

The above proof holds for any $r$, but it is only tight for $r\geq d+1$.

For $d>r$, a subset $K$ of $B_t(x)$ is \emph{$(d,r)$-canonical} if there is a subset $I$ of $[d]$ of size $2d-r+1$ and $\epsilon_i\in\{-1,1\}$ for each $i\in I$ such that
\[
K = \{y\in B_t(x) : y_i-x_i=\epsilon_i \text{ for all $i\in I$ and } y_i=x_i \text{ otherwise}\}.
\]

\begin{lemma}
Let $3\leq d+1\leq r\leq 2d$ and $t\geq 0$. Suppose the origin is protected and $P(B_t)$ is minimal. Then $P(B_t)$ is $(d,r)$-canonical.
\end{lemma}

We omit the proof since it is similar to the proof of Theorem \ref{th:extremal}.

\begin{theorem}\label{th:sub}
Let $3\leq d+1\leq r\leq 2d$ and $t\geq 0$. Let $(p_n)_{n=1}^\infty$ be a sequence of probabilities and let $\omega(n)\to\infty$.
\begin{enumerate}
\item If $t\leq 2d-r+1$ and
\[
n^{-d/m_{d,r}(t-1)}\omega(n) \leq q_n \leq n^{-d/m_{d,r}(t)}/\omega(n),
\]
then $T=t$ with high probability.
\item If $t\leq 2d-r+1$ and $q_n^{m_{d,r}(t)}n^d \to c$ as $n\to\infty$, for some constant $c>0$, then $T\in\{t,t+1\}$ with high probability.
\item If
\begin{equation}\label{eq:qlarge}
q_n \geq n^{-d/m_{d,r}(2d-r+1)}\omega(n) = n^{-d/2^{2d-r+1}}\omega(n),
\end{equation}
then $T=\infty$ with high probability.
\end{enumerate}
\end{theorem}

No stability theorem is needed for the proof of Theorem \ref{th:sub} because if percolation occurs then the assertion is that it occurs in a time that does not depend on $n$, so we only ever need to consider configurations inside $\ell_1$ balls of bounded size.

\begin{proof}
Parts (i) and (ii) of the theorem follow using the same methods as Theorem \ref{th:main}, so we only have to prove (iii). Let $m=2d-r+1$. It is sufficient to show that if $q_n$ satisfies \eqref{eq:qlarge} then with high probability $\T_n^d$ contains an uninfected $m$-dimensional hypercube at time $0$. More specifically, if all the sites in some translate of
\[
\{0,1\}^{m} \times \{0\}^{d-m}
\]
are initially uninfected, then they remain uninfected forever, so it suffices to prove that there exists an empty such translate with high probability. There are $n^d2^{-m}$ disjoint hypercubes of the given form, and the probability that none is initially uninfected is
\[
\big(1-q_n^{2^m}\big)^{n^d 2^{-m}} \leq \exp\big(-q_n^{2^m} n^d 2^{-m}\big) \leq \exp\big(-2^{-m}\omega(n)^{2^m}\big),
\]
which is $o(1)$.
\end{proof}

Finally, we note the corollary mentioned at the beginning of the section, which concerns the critical probabilities of subcritical models. Let $P(d,r,n,p)$ be the probability that a random set $A\subset \T_n^d$ percolates under the $r$-neighbour process, where sites are included in $A$ independently with probability $p$.

\begin{corollary}
Let $3\leq d+1\leq r\leq 2d$ and let $(p_n)_{n=1}^\infty$ be a sequence of probabilities.
\begin{enumerate}
\item If $q_n = O(n^{-d/m_{d,r}(2d-r+1)})$, then $P(d,r,n,p_n) \to 1$ as $n\to\infty$. \\
\item If $q_n \gg n^{-d/m_{d,r}(2d-r+1)}$, then $P(d,r,n,p_n) \to 0$ as $n\to\infty$. \qed
\end{enumerate}
\end{corollary}

\bibliographystyle{amsplain}
\bibliography{../bprefs}

\end{document}